\documentclass[11pt,a4paper]{amsart}

\usepackage[T1]{fontenc}
\usepackage{amsmath, amssymb, graphicx, amsthm, amsfonts}
\usepackage{a4wide}
\usepackage{graphpap}
\usepackage{hyperref, url}

\newtheorem{thm}{Theorem}

\theoremstyle{plain}

\newtheorem{theorem}{Theorem}[subsection]
\newtheorem{lemma}[theorem]{Lemma}
\newtheorem{proposition}[theorem]{Proposition}
\newtheorem{corollary}[theorem]{Corollary}

\theoremstyle{definition}

\newtheorem{definition}[theorem]{Definition}

\theoremstyle{remark}

\newtheorem{example}[theorem]{\it Example\/}
\newtheorem{question}[theorem]{Question}

\newcommand{\Z}{\mathbb Z}

\newcommand{\R}{\mathbb R}
\newcommand{\C}{\mathbb C}

\newcommand{\Sph}{\mathbb{S}}

\newcommand\lk{\mathrm{lk}}
\newcommand{\mama}[4]{\left( \begin{matrix} #1 & #2 \\ #3 & #4 \end{matrix} \right)}

\newcommand{\wt}{\widetilde}

\newcommand{\bord}{\partial}

\newcommand{\MuraOrdre}{\prec}
\newcommand{\MuraOrdreLarge}{\preceq}
\newcommand{\Hopf}{\mathcal H}

\newcommand{\surf}{\Sigma}

\newcommand{\rect}[2]{R_{#1}^{#2}}

\newcommand{\donnepar}[1]{\overset{\tau#1}{\longmapsto}}
\newcommand{\col}{c}
\newcommand{\bbot}{b}
\newcommand{\ttop}{t}
\newcommand{\rec}{r}
\newcommand{\botcol}{b_c}
\newcommand{\topcol}{t_c}
\newcommand{\botrec}{b_{r}}
\newcommand{\toprec}{t_{r}}
\newcommand{\botrecc}{b_{r{+}1}}
\newcommand{\toprecc}{t_{r{+}1}}

\newcommand{\lray}[1]{R^\nwarrow_{#1}}
\newcommand{\rray}[1]{R^\nearrow_{#1}}
\newcommand{\vray}[1]{R^\downarrow_{#1}}
\newcommand{\topcell}[1]{t(#1)}
\newcommand{\botcell}[1]{b(#1)}

\newcommand{\acces}{A}
\newcommand{\accesg}{B}
\newcommand{\accesd}{C}

\newcommand{\carre}{\Omega}
\newcommand{\carreb}{\Omega'}
\newcommand{\cy}[1]{[#1]}
\newcommand{\cybis}[1]{[#1]_\sigma}

\newcommand{\DT}{\tau}
\newcommand{\DTgeom}{\tau}

\newcommand{\equerre}[1]{E_{#1}}

\newcommand{\MurasommeB}{\mathbin{\mathtt{\#}}}

\newcommand{\surfd}{\Sigma_D}
\newcommand{\seifbis}{\wt{\Sigma_D}}

\begin{document}

\title{On the zeroes of the Alexander polynomial of a Lorenz knot}

\author{Pierre Dehornoy}

\subjclass[2000]{Primary 57M27; Secondary 34C25, 37B40, 37E15, 57M25}
\keywords{Lorenz knot, Alexander polynomial, monodromy, surface homeomorphism}

\date{\today}

\begin{abstract}
We show that the zeroes of the Alexander polynomial of a Lorenz knot all lie in some annulus whose width depends explicitly on the genus and the braid index of the considered knot.
\end{abstract}

\maketitle 


Lorenz knots~\cite{BW} are a family of knots that arise in the context of dynamical systems as isotopy classes of periodic orbits of the Lorenz flow~\cite{Lorenz}, a particular flow in~$\R^3$. They received much attention in the recent years because they form a relatively large family that includes all torus knots and all algebraic knots and, at the same time, they are not too complicated and their geometric origin in dynamical systems provides specific tools to study them~\cite{BK, EM, GhysLorenz}. On the other hand, the Alexander polynomial is a classical knot invariant, that is, a polynomial which only depends on the topological type of the knot. It is known that any polynomial~$\Delta$ with integer coefficients that is symmetric, in the sense that the inverse of every zero is also a zero, and sastisfying~$\vert\Delta(1)\vert=1$, is the Alexander polynomial of at least one knot~\cite{Kawauchi}. Therefore it seems hard to expect much in the direction of controlling the zeroes of the Alexander polynomial of an arbitrary knot. By contrast, the result we prove in this paper asserts that, in the case of a Lorenz knot, the zeroes of the Alexander polynomial must lie in some definite annulus depending on the genus of the knot (that is, the smallest genus of a surface spanning the knot) and the braid index (that is, the smallest number of strands of a braid whose closure is the knot):

\begin{thm}
\label{Theoreme}
Let $K$ be a Lorenz knot. Let $g$ denote its genus and $b$ its braid index. Then the zeroes of the Alexander polynomial of~$K$ lie in the annulus
\begin{equation*}
\left\{ z\in\C \,\big\vert\, (2g)^{-4/(b-1)} \le \vert z \vert \le (2g)^{4/(b-1)} \right\}.
\end{equation*}
\end{thm}

This implies in particular that, among the Lorenz knots that can be represented by an orbit of length at most~$t$, the proportion of knots for which all zeroes of the Alexander polynomial lie in an annulus of diameter~$O(t^{c/t})$ tends to~$1$ when $t$ goes to infinity (Corollary~\ref{Coro}).
 
The possible interest of Theorem~\ref{Theoreme} is double.
First, it provides an effective, computable criterion for proving that a knot is not a Lorenz knot (Corollary~\ref{C:NotLorenz}). 

Second, Theorem~\ref{Theoreme} may be seen as a first step in the direction of understanding Alexander polynomials of orbits of general flows. Given a flow~$\Phi$ in~$\R^3$, it is natural to look at its periodic orbits as knots, and to wonder how these knots caracterize the flow \cite{GHS}. Let us call~$k(x,t)$ the piece of length~$t$ of the orbit of~$\Phi$ starting at~$x$, closed with the geodesic segment connecting~$\Phi^t(x)$ to~$x$. Then~$k(x,t)$ is a loop. In most cases, this loop has no double points, thus yielding a knot. Arnold~\cite{Arnold} studied the linking number of two such knots. In the case of an ergodic volume-preserving vector field, he showed that the limit~$\lim_{t_1, t_2\to\infty} \lk(k(x_1,t_1), k(x_2, t_2))/t_1t_2$ exists and is independent of the points~$x_1, x_2$, thus yielding a topological invariant for the flow. It turns out that this knot-theoretical invariant coincides with the helicity of the vector field. Later, Gambaudo and Ghys in the case of $\omega$-signatures \cite{GamGhys} and Baader and March\'e in the case of Vassiliev invariants~\cite{BaaderMarche} established similar asymptotic behaviours, with all involved constants proportional to helicity. It is then natural to wonder whether other knot-theoretical invariants have analogous behaviours, and, if so, whether the constants are connected with the helicity. For instance, numerical experiments suggest that the $3$-genus might obey a different scheme, but no proof is known so far. On the other hand, the Alexander polynomial is a sort of intermediate step between signatures and genus: its degree is bounded from below by all signatures, and from above by twice the genus. Therefore, controlling the asymptotic behaviour of the Alexander polynomial and its zeroes is a natural task in this program. It is known that the zeroes on the unit circle are determined by the collection of all~$\omega$-signatures, but nothing was known for other zeroes, and this is what Theorem~\ref{Theoreme} provides, in the case of Lorenz knots.

The principle of the proof of Theorem~\ref{Theoreme} consists in interpreting the modulus of the largest zero of the Alexander polynomial of a Lorenz knot as the growth rate of the associated homological monodromy. 
More precisely, as every Lorenz knot~$K$ is the closure of a positive braid of a certain type~\cite{BW}, we start from the standard Seifert surface~$\Sigma$ associated with this braid. As the involved braid is necessarily positive, $\Sigma$ can be realized as an iterated Murasugi sum~\cite{Murasugi} of positive Hopf bands. Then, we interpret the Alexander polynomial of~$K$ as the characteristic polynomial of the homological monodromy~$h_*$ of~$K$, an endomorphism of the first homology group~$H_1(\Sigma; \Z)$, which is well defined because $K$ is fibered with fiber~$\Sigma$. From here, our goal is then to bound the growth rate of~$h_*$. To this end, we use the decomposition of~$\Sigma$ as an iterated Murasugi sum to express the geometric monodromy of~$K$ as a product of positive Dehn twists, and we deduce an expression of the homological monodromy~$h_*$ as a product of transvections. The hypothesis that the knot is a Lorenz knot implies that the pattern describing how the Hopf bands are glued in the Murasugi decomposition of~$\Sigma$ is very special. By using this particularity and choosing a (tricky) adapted basis of~$H_1(\Sigma; \Z)$, we control the growth of the $\ell^1$-norm of a cycle when the monodromy is iterated. Finally, the bound on the $\ell^1$-norm induces a bound on the eigenvalues of ~$h_*$, and, from there, a bound on the zeroes of the Alexander polynomial of~$K$.

It may be worth noting that our main argument is more delicate than what one could a priori expect. Indeed, using the standard Murasugi decomposition of the Seifert surface, which is obtained by attaching all disks behind the diagram (Figure~\ref{F:SurfaceStandard}), cannot work for our purpose. Instead we must consider a non-standard decomposition also obtained by applying the Seifert algorithm, but by attaching half of the disks in front of the diagram and half of the disks behind (Figure~\ref{F:RadiateurLorenz}).

As suggested by the above sketch of proof, Theorem~\ref{Theoreme} can be interpreted in terms of growth rate of surface homeomorphisms. Namely, if $K$ is a Lorenz knot with Seifert surface~$\Sigma$ and monodromy~$h$, then what we do is to control the growth rate of the induced action~$h_*$ on homology. If one consider directly the action of~$h$ on curves on~$\Sigma$, then Thurston~\cite{FLP, Thurston} defined a number that control how curves are stretched by~$h$. It is called the \emph{dilatation} of~$h$. The dilatation has been the subject of intense studies, and in particular determining the minimal possible dilatation on a surface of fixed genus is still an open problem~\cite{BBK, Hironaka2, HironakaKin, LT, Penner}. In general, the homological growth rate is smaller than the dilatation, so that our main result has no consequences related to the dilatation. However, as an important tool of our proof (Lemma~\ref{T:ImageInterieure}) holds also for curves, we formulate a similar conjecture for the dilatation, see Section~\ref{S:Questions}.

Computer experiments played an important role during the preparation of this paper. Propositions~\ref{T:DecompositionMurasugi} and~\ref{T:DecompositionMurasugiBis} below lead to an algorithm for computing the homological monodromy of Lorenz knots, and we ran it on large samples of thousands of knots. Using Bar-Natan's package KnotAtlas\footnote{\url{http://katlas.org/}} to double-check the value of the Alexander polynomial, we obtained strong evidence for the formulas of Sections~\ref{S:FirstPass} and~\ref{S:SecondPass} before their proof was completed. Also, the choice of the surface~$\seifbis$ in Section~\ref{S:MixedSurface} was directly inspired by the computer experiments.

The plan of the paper is as follows. In Section~\ref{S:Preliminaries}, we recall the definitions of Lorenz knots, Lorenz braids, and the associated Young diagrams. Then we describe Murasugi sums, and explain how they preserve fiberedness and compose monodromies. Finally, we construct for every Lorenz knot a standard Seifert surface using an iterated Murasugi sum of Hopf bands, and deduce an explicit formula for the monodromy. In Section~\ref{S:Combinatorics}, starting from the standard decomposition of the Seifert surface, we first develop a combinatorial analysis of the homological monodromy, and explain what is missing to derive a bound for the growth rate. Then we consider another Murasugi decomposition, and show how to adapt the combinatorial analysis of the monodromy. In Section~\ref{S:Radius}, we use the latter analysis for bounding the eigenvalues of the monodromy, thus proving Theorem~\ref{Theoreme}. We then give some examples and conclude with a few questions and further observations.

\thanks{I thank \'Etienne Ghys for many enlightening discussions, Hugh Morton, who taught me the basic material of this article, in particular the Murasugi sum, during a visit at Liverpool, Joan Birman and the anonymous referee for many remarks and corrections.}

\section{Preliminaries}
\label{S:Preliminaries}

The aim of this section is to express the homological monodromy of every Lorenz knot as an explicit product of transvections (Proposition~\ref{T:FormuleMonodromie}).

It is organized as follows. We first recall the basic definitions about Lorenz knots starting from Young diagrams. Then, we describe the Murasugi sum  in Section~\ref{S:Topologie} and the iterated Murasugi sum in Section~\ref{S:MurasugiSum}. Finally, we use the Murasugi sum in Section~\ref{S:StandardSurface} to give a geometric construction of the Seifert surface associated to a Lorenz knot and derive the expected expression of the homological monodromy.


\subsection{Lorenz knots, Lorenz braids, and Young diagrams}
\label{S:Young}

Lorenz knots and links were introduced by Birman and Williams~\cite{BW} as isotopy classes of sets of periodic orbits of the geometric Lorenz flow~\cite{Lorenz} in~$\R^3$. They are closure of Lorenz braids. It is explained in~\cite{EM} how to associate a Young diagram with every Lorenz braid. Here we shall go the other way and introduce Lorenz braids starting from Young diagrams.

\begin{definition}
\label{D:Lorenz}
Let $D$ be a Young diagram, supposed to be drawn as in Figure~\ref{F:Young} left; extend the edges both up and down so that it looks like the projection of a braid, orient the strands from top to bottom, and desingularize all crossings positively. The braid~$b_D$ so obtained (Figure~\ref{F:Young} right) is called the~\emph{Lorenz braid} associated with~$D$, and its closure~$K_D$ is called the \emph{Lorenz knot} associated with~$D$.
\end{definition}

\begin{figure}[htbp]
	\includegraphics[width=0.8\textwidth]{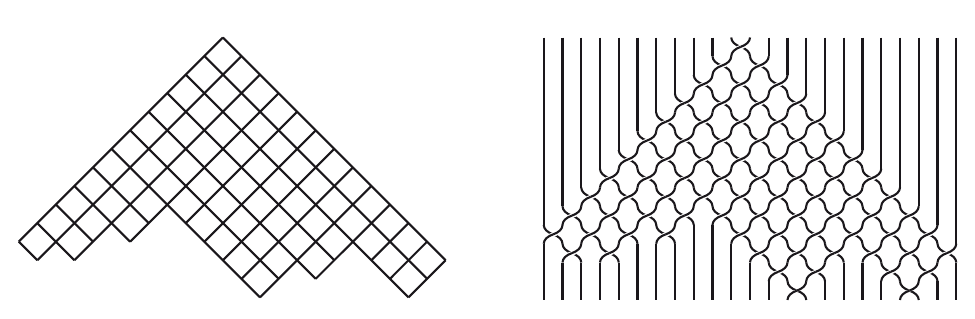}
	\caption{\small How to transform a young diagram into a Lorenz braid.}
	\label{F:Young}
\end{figure}

\begin{example} Consider the Young diagram with columns of heights~$2,1,1$ respectively. Then the associated Lorenz braid is~$\sigma_4\sigma_3\sigma_5\sigma_2\sigma_4\sigma_6\sigma_1\sigma_3\sigma_5\sigma_2$. Its closure turns out to be the $(5,2)$-torus~knot, which is therefore a Lorenz knot.
\end{example}

It may happen that the closure of a Lorenz braid has more than one component, and should therefore be called a Lorenz \emph{link}, instead of a knot. Many properties of Lorenz knots are shared by Lorenz links, but their complement can admit several non isotopic fibrations. This is a problem for our approach. Therefore, in the sequel, we always implicitely refer to Young diagrams and Lorenz braids which give rise to Lorenz knots, and not to Lorenz links.

Let us introduce some additional notation. Let $D$ be a Young diagram. We give coordinates to cells (see Figure~\ref{F:Coordinates}) by declaring the top cell to be $(0,0)$, by adding~$(-1,1)$ when going on an adjacent $SW$-cell, and by adding~$(1,1)$ when going on an adjacent $SE$-cell. Thus coordinates always have the same parity. The \emph{$c$th column} consists of the cells whose first coordinate is~$c$. Integers~$t_c, b_c$ are defined so that $(c, t_c)$ is the top cell, and~$(c, b_c)$ the bottom cell, of the $c$th column. Observe that we always have $t_c= \vert c \vert$. The column on the left of the diagram is denoted by~$c_l$. Observe that it contains the cell~$(c_l, -c_l)$ only. Similarly the column on the right is denoted by~$c_r$, and it contains the cell~$(c_r, c_r)$ only.

\begin{figure}[htbp]
	\begin{center}
	\begin{picture}(180,130)(0,0)
	\put(0,0){\includegraphics[width=0.4\textwidth]{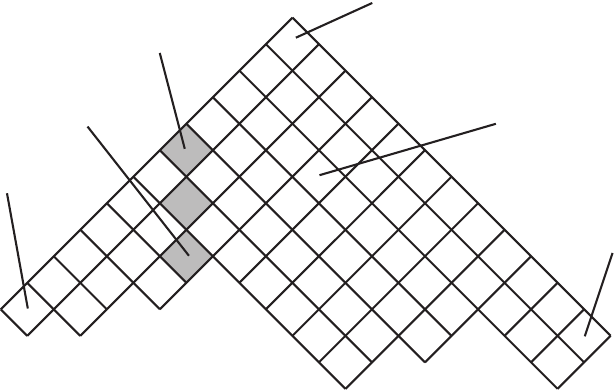}}
	\put(109,109){$(0,0)$}
	\put(146,75){$(1,5)$}
	\put(162,43){$(11,11)$}
	\put(-20,60){$(\col_l,-\col_l)$}
	\put(10,79){$(\col,\bbot_\col)$}
	\put(35,102){$(\col,\ttop_\col)$}
	\end{picture}
	\end{center}
	\caption{\small Coordinates in a Young diagram. The $c$th column, with~$c=-4$, is in grey.}
	\label{F:Coordinates}
\end{figure}
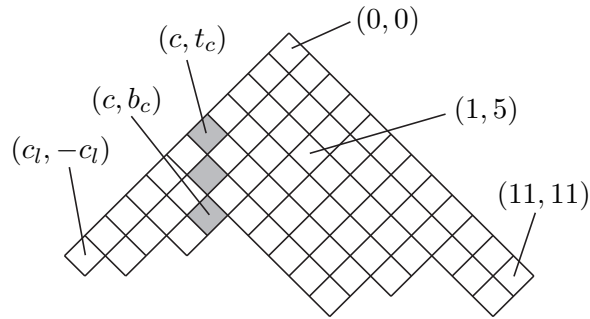


\subsection{Murasugi sum, fibered links and monodromy}
\label{S:Topologie}

By definition, Lorenz knots are closures of positive braids. An important consequence is that they are fibered~\cite{BW}, and that the monodromy homeomorphism is a product of positive Dehn twists. In order to understand and use these properties, we recall a simple and very geometric operation: the Murasugi sum~\cite{Murasugi, Gabai1, Gabai2}. The idea is to iteratively construct the fibration of the complement of a knot by adding the crossings of the braid one by one. For this, we use two-component Hopf links as building blocks, and the Murasugi sum as a gluing tool. 

From now on, we work in the sphere~$\Sph^3$, identified with $\R^3\cup\{\infty\}$. 

\begin{definition}
\label{D:Hopf}
~

\begin{itemize}
\item[($i$)] A~\emph{positive Dehn twist} is a map from $[0,1]\times\Sph^1$ into itself isotopic to~$\DTgeom$ defined by~$\DTgeom(r, \theta) = (r, \theta+r)$. 
\item[($ii$)] Let~$\Sigma$ be a surface and~$\gamma$ be an immersed smooth curve in~$\Sigma$. Consider a tubular neighbourhood~$A$ of~$\gamma$ in~$\Sigma$, and parametrize it by~$[0,1]\times\Sph^1$ so that the orientations coincide. A \emph{positive Dehn twist along $\gamma$} is the class of the homeomorphism~$\DT_\gamma$ of~$\Sigma$ that coincides with a positive twist of the annulus~$A$ and that is the identity outside. 
\item[($iii$)] By extension, A positive Dehn twist along~$\gamma$ is the induced automorphism~$\DT_\gamma$ of the module~$H_1(\Sigma, \bord\Sigma; \Z)$.
\end{itemize}
\end{definition}

	When the surface~$\Sigma$ in an annulus, a natural basis for~$H_1(\Sigma, \bord\Sigma; \Z)$ is made of the core of the annulus, and a transversal radius. Then, the matrix of a positive Dehn twist is~$\mama 1 1 0 1$, so that the homological twist is a transvection (Figure~\ref{F:ToreHopf} right).

\begin{proposition}
\label{T:MonodromieHopf}
	The complement of a positive, two-component Hopf link in~$\Sph^3$ fibers over~$\Sph^1$, the fiber being an annulus and the monodromy a positive Dehn twist.
\end{proposition}

\begin{proof}
We use Figure~\ref{F:ToreHopf} for the proof: 
on the left, a positive Hopf link is depicted as the boundary of an annulus, both being drawn on the boundary of a solid torus. 
In the center left, we see one half of the monodromy, corresponding to what happens on one meridian disk inside the solid torus. 
Since the complement of the solid torus in~$\Sph^3$ is another solid torus---meridians and parallels being exchanged---the monodromy is the composition of the map from the green annulus to the white one, and of its analog from the white annulus to the green one obtained by a $90^\circ$-rotation. 
It is the positive Dehn twist depicted on the center right. 
The action on cycles is displayed on the right: the core (in green) remains unchanged, while the radius (in red) is mapped on a curve winding once along the core (in orange).
\end{proof}

\begin{figure}[htb]
	\begin{center}
	\includegraphics[width=0.6\textwidth]{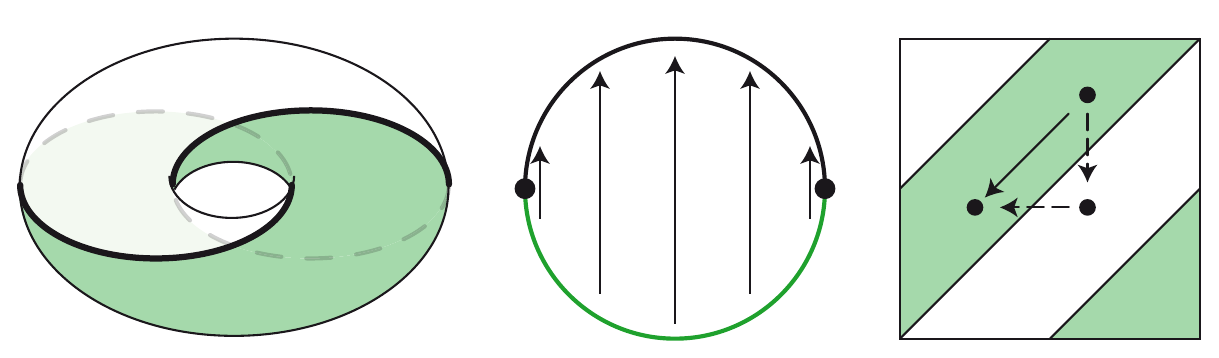}\quad
	\includegraphics[width=0.17\textwidth]{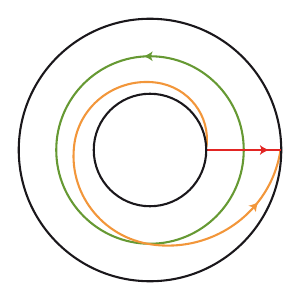}
	\end{center}
	\caption{\small On the left, a positive Hopf link as the boundary of an annulus. In the center left, one half of the monodromy. On the right, the action of the monodromy on cycles.}
	\label{F:ToreHopf}
\end{figure}

\begin{definition}
\label{D:sommemurasugi}
	(See Figure~\ref{F:sommeMurasugi}.)
	Let $\surf_1$ and $\surf_2$ be two oriented surfaces embedded in $\Sph^3$ with respective boundaries~$K_1$ and $K_2$. 
	Let~$\Pi$ be an embedded sphere (seen as the horizontal plane in~$\R^3\cup\{\infty\}$). Call~$B_1$ and $B_2$ the open balls that~$\Pi$ separates. 
	Suppose that
\begin{itemize}
\item[($i$)] the surface $\surf_1$ is included in the closure of the ball~$B_1$, and $\surf_2$ in the closure of~$B_2$;
\item[($ii$)] the intersection $\surf_1\cap\surf_2$ is a $2n$-gon, denoted~$P$, contained in~$\Pi$ with the orientations of $\surf_1$ and~$\surf_2$ on~$P$ coinciding and pointing into~$B_2$;
\item[($iii$)] the links $K_1$ and $K_2$ intersect at the vertices of~$P$, that we denote by~$x_1, \dots, x_{2n}$.
\end{itemize}
We then define the \emph{Murasugi sum $\surf_1\MurasommeB_P\surf_2$ of $\surf_1$ and $\surf_2$ along $P$} as their union $\surf_1\cup\surf_2$. 
	We define the \emph{Murasugi sum $K_1\MurasommeB_P K_2$ of $K_1$ and~$K_2$ along~ $P$} as the link~$K_1\cup K_2 \smallsetminus \bigcup\, ] x_i , x_{i+1} [$.
\end{definition}

More generally, we define the Murasugi sum of two disjoint surfaces~$\surf_i, i=1,2$ along two polygons~$P_i$ with one specified vertex as the isotopy class of the Murasugi sum of two isotopic copies of~$\surf_i$ respecting conditions~$(i), (ii)$ and $(iii)$ of Definition~\ref{D:sommemurasugi} and such that the polygons~$P_i$ and the specified vertices coincide. As we might expect, this surface is unique up to isotopy. We denote it by~$\surf_1\MurasommeB_{P_1\sim P_2}\surf_2$.

\begin{figure}[htbp]
	\begin{center}
	\includegraphics[scale=0.3]{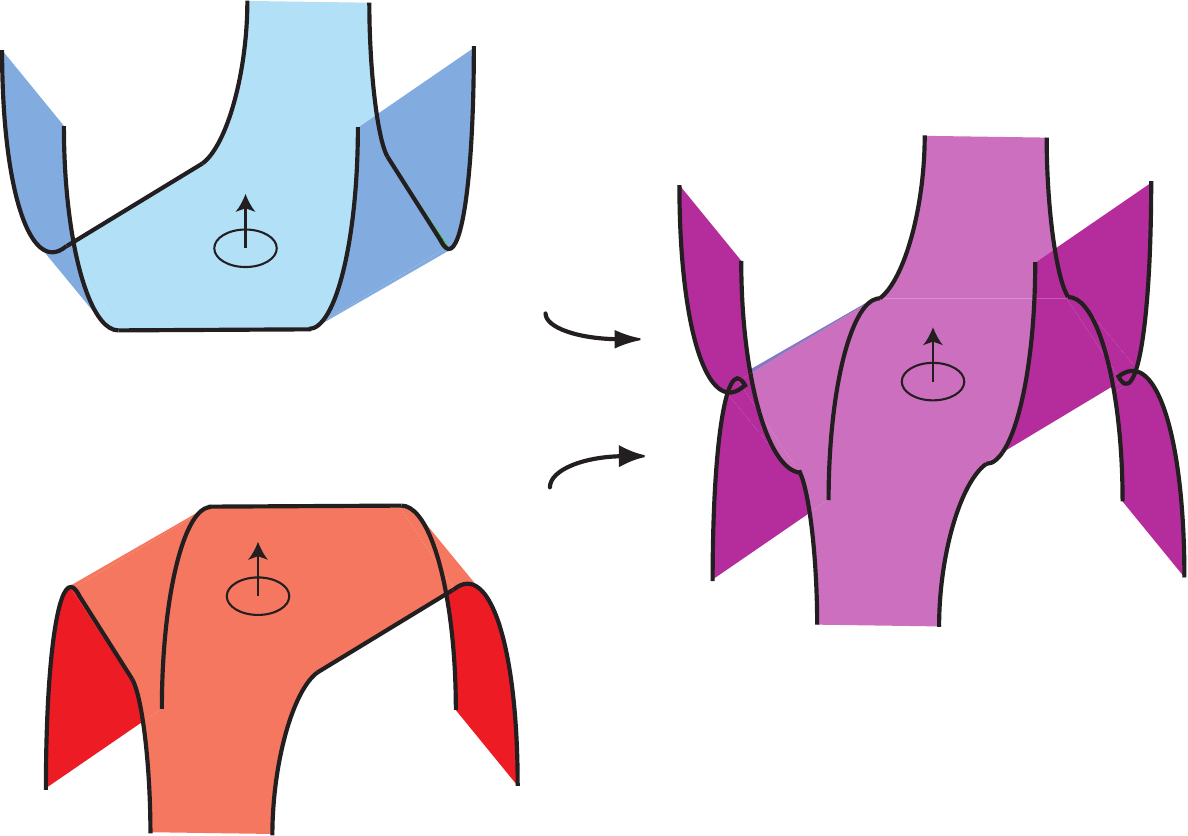}
	\end{center}
	\caption{\small The Murasugi sum of two surfaces with boundary.}
	\label{F:sommeMurasugi}
\end{figure}

The Murasugi sum generalizes the connected sum, which corresponds to the case~$n=1$ in the definition. It also generalizes plumbing, which corresponds to~$n=2$. It is a natural geometric operation for surfaces---and for the links they bound---in the sense that it preserves important properties, like, for instance, being incompressible, being a minimal genus spanning surface, or being a fibered link (see~\cite{Gabai1, Gabai2, EM} and below).

\begin{theorem}
	\label{T:sommefibre}
	Let~$K_1$ and $K_2$ be two fibered links in~$\Sph^3$ with respective fibers~$\surf_1$ and~$\surf_2$. Let~$h_1$ and~$h_2$ be the class of their respective geometric monodromies. Let $P_1$ ({\it resp.\ }$P_2$) be a $2n$-gon on~$\surf_1$ ({\it resp.\ }$\surf_2$) whose even ({\it resp.\ }odd) edges are included in the boundary~$K_1$ ({\it resp.\ }$K_2$) of $\surf_1$ ({\it resp.\ }$\surf_2$). Then
\begin{itemize}
\item[($i$)] the Murasugi sum $K_1\MurasommeB_{P_1\sim P_2} K_2$ is fibered with fiber $\surf_1\MurasommeB_{P_1\sim P_2}\surf_2$;
\item[($ii$)] the monodromy of~$K_1\MurasommeB_{P_1\sim P_2} K_2$ is~$h_1\circ h_2$, where $h_1$ ({\it resp.\ }$h_2$) is extended as an application of~$\surf_1\MurasommeB_{P_1\sim P_2}\surf_2$ by the identity on the complement~$\Sigma_2\smallsetminus\Sigma_1$ ({\it resp.}~$\Sigma_1\smallsetminus\Sigma_2$).
\end{itemize}
\end{theorem}

\begin{proof}[Proof (sketch, see~\cite{EM} for details)]
First apply an isotopy to the links~$K_1, K_2$ and to the surfaces~$\surf_1, \surf_2$ in order to place them in a good position, namely place $K_1$ and~$\surf_1$ in the upper half space, and $K_2$ and~$\surf_2$ in the lower half space (Figure~\ref{F:sommeMurasugi}). Then zoom on the neighbourhood of~$P_1$ ({\it resp.}~$P_2$) and rescale time so that the fibration, denoted~$\theta_1$ ({\it resp.~$\theta_2$}), of the complement of~$K_1$ ({\it resp.}~$K_2$) on the circle becomes trivial in the lower half space ({\it resp.\ }upper half space) and takes time~$[0,\pi]$ ({\it resp.}~$[\pi, 2\pi]$), see Figure~\ref{F:fibrationreparametree}. Finally consider the function~$\theta$ of the complement of~$K_1\MurasommeB_{P_1\sim P_2} K_2$ which is equal to~$\theta_1$ on the upper half space, to~$\theta_2$ on the lower half space (Figure~\ref{F:collagefibration}), and is defined according to Figure~\ref{F:cylindre} around the sides of~$P$. Check that~$\theta$ has no singularity and that the $0$-level is~$\surf_1\MurasommeB_{P_1\sim P_2}\surf_2$. Then $\theta$ induces fibration over the circle of the complement of~$K_1\MurasommeB_{P_1\sim P_2} K_2$ with fiber~$\surf_1\MurasommeB_{P_1\sim P_2}\surf_2$. As for the monodromy, a curve on~$\surf_1\MurasommeB_{P_1\sim P_2}\surf_2$ is first transformed in the lower half space according to~$h_2$, and then in the upper half space according to~$h_1$, so that the monodromy is the composition.
\end{proof}

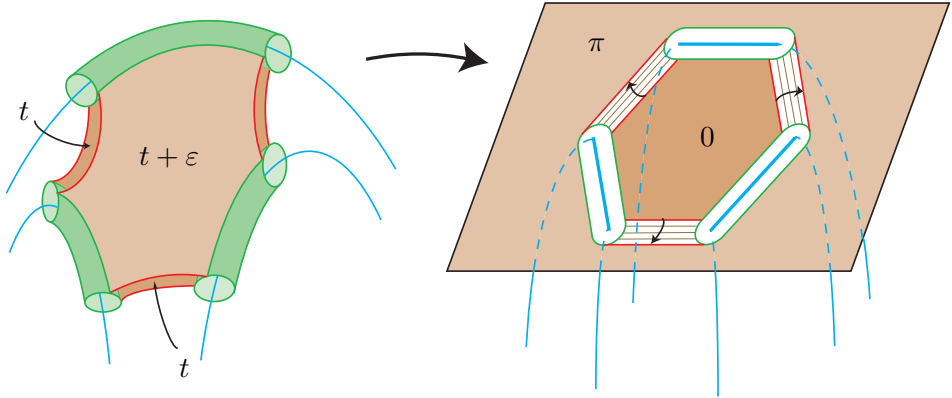
\begin{figure}[htbp]
	\begin{center}
	\begin{picture}(330,170)(0,0)
	\put(0,0){\includegraphics[width=0.8\textwidth]{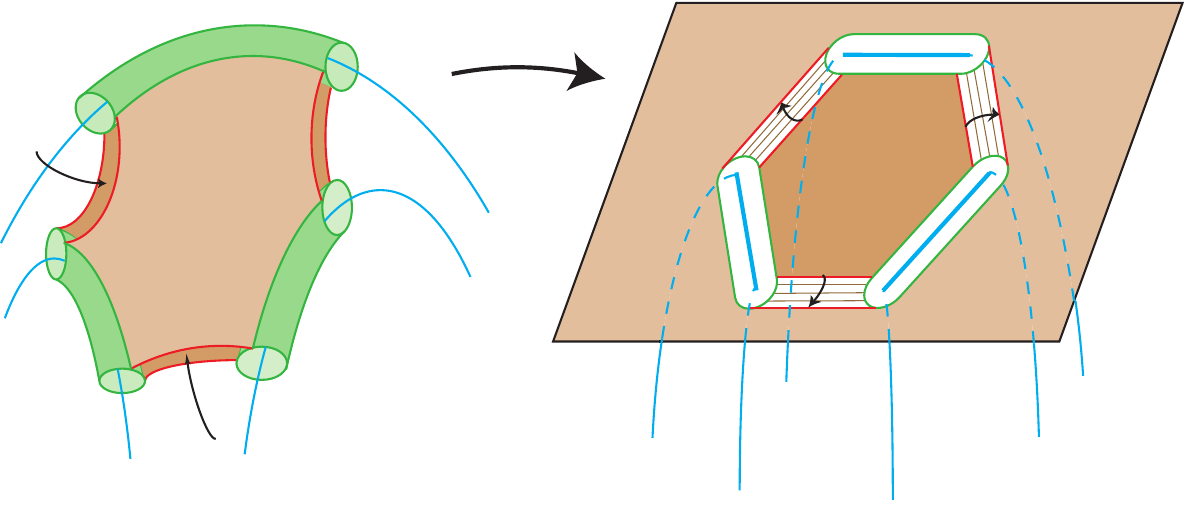}}
	\put(5,107){$t$} 
	\put(65,10){$t$}
	\put(50,90){$t+\varepsilon$}
	\put(260,97){$0$} 
	\put(218,133){$\pi$}
	\end{picture}
	\end{center}
	\caption{\small Deformation of the fibration of~$K_2$ so that it becomes trivial in the upper half space. It is obtained by zooming on a small neighbourhood of the polygon~$P$}
	\label{F:fibrationreparametree}
\end{figure}

\begin{figure}[htbp]
	\begin{center}
	\includegraphics[width=0.7\textwidth]{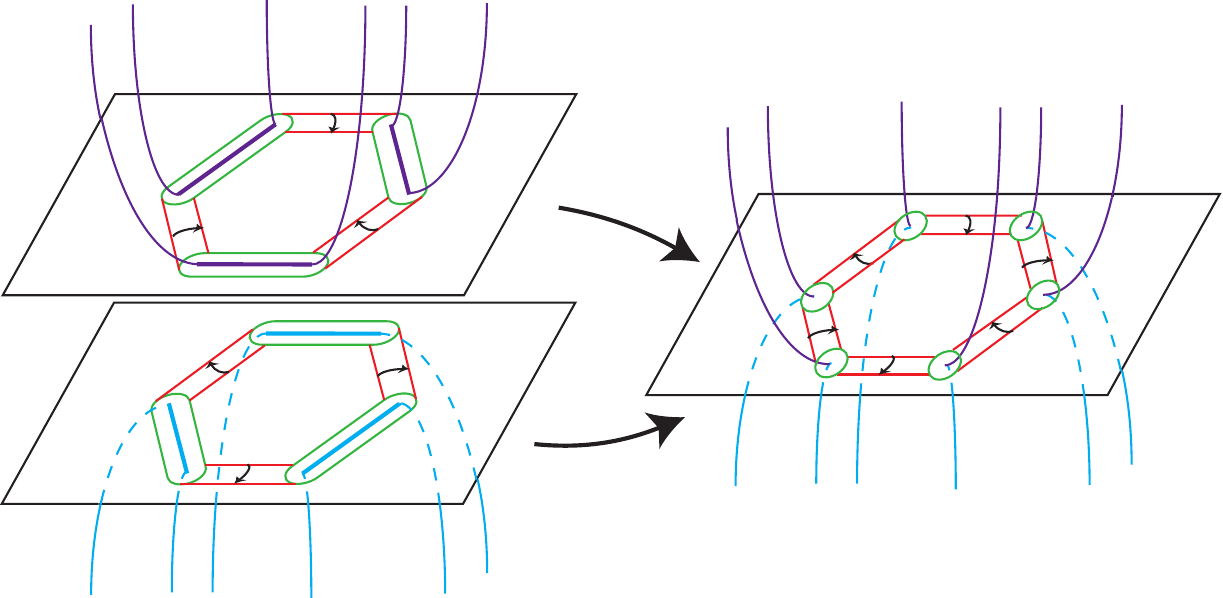}
	\end{center}
	\caption{\small Global picture for the gluing of two fibrations in good positions in order to obtain a fibration for the Murasugi sum.	
	}
	\label{F:collagefibration}
\end{figure}

\begin{figure}[htb]
	\begin{center}
	\includegraphics[scale=0.6]{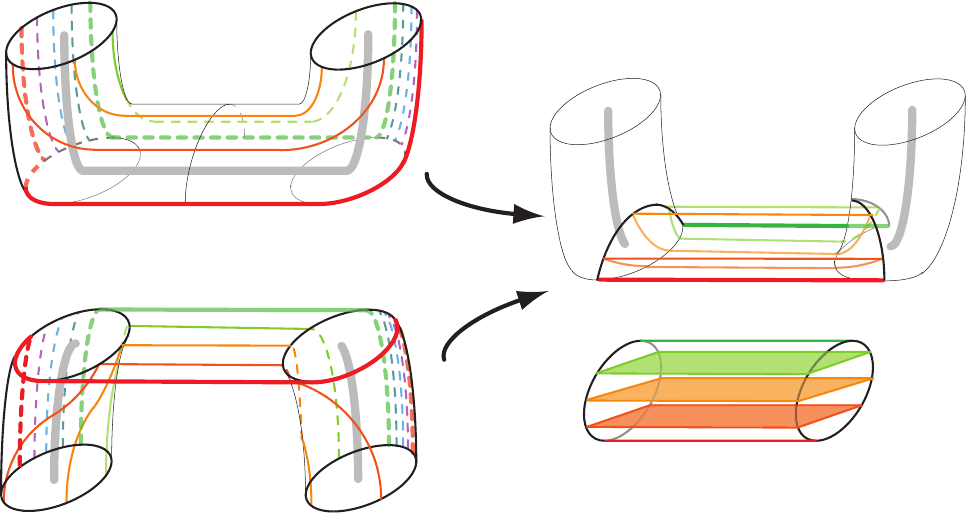}
	\end{center}
	\caption{\small How to define the fibration~$\theta$ around the sides of the polygon~$P$. On the left, the levels of the fibrations~$\theta_1$ and~$\theta_2$ in their respective half spaces. On the top right, the part on which~$\theta$ get a new definition. It is a union of disjoint cylinders with prescribed boundary values. On the bottom right, a foliation of a cylinder obeying this constraint.}
	\label{F:cylindre}
\end{figure}

\subsection{Iterated Murasugi sum}
\label{S:MurasugiSum}

We can now glue several fibered links together. In order to obtain a decomposition for the monodromy, we have to keep track on the order of the gluing operations, and on the top/bottom positions of the surfaces. A first example is displayed on Figure~\ref{F:TrefleFibre}, showing that a Murasugi sum of two Hopf bands yields a Seifert surface and a fibration for the trefoil knot. We can then iterate, and see that the closure of the braid~$\sigma_1^n$ is the Murasugi sum of~$n-1$ Hopf bands, each of them associated to two consecutive crossings. The monodromy of the resulting link is the product of~$n-1$ Dehn twists along the cores of the bands, performed starting from the bottom to the top of the braid. Then, by gluing two braids side by side as displayed on Figure~\ref{F:CollageLateral}, one obtains more complicated knots.

\begin{figure}[htbp]
	\begin{center}
	\includegraphics[width=0.55\textwidth]{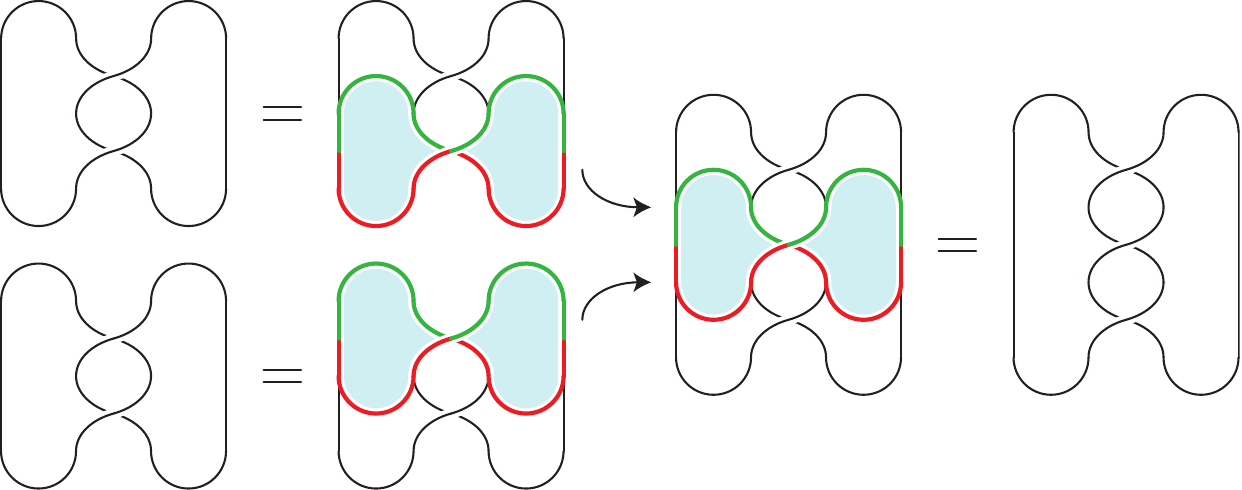}
	\end{center}
	\caption{\small How to glue two Hopf bands and obtain a fibration for the closure of the braid~$\sigma_1^n$.}
	\label{F:TrefleFibre}
\end{figure}

\begin{figure}[htbp]
	\begin{center}
	\includegraphics[width=0.45\textwidth]{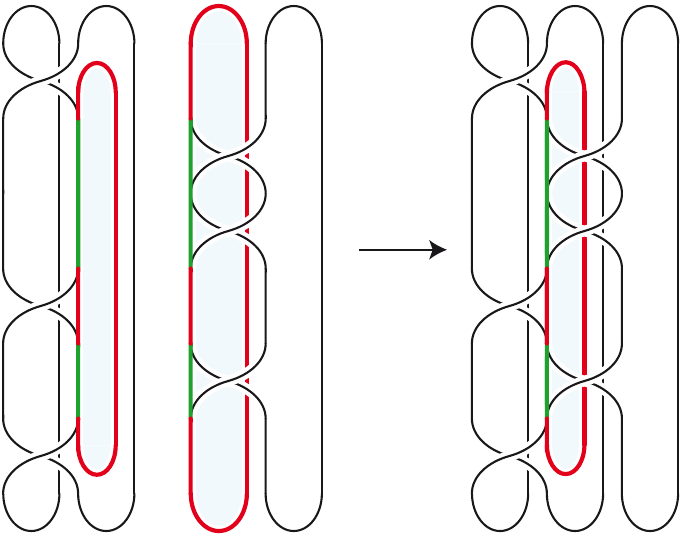}
	\end{center}
	\caption{\small The Murasugi sum of the closures of two positive braids.}
	\label{F:CollageLateral}
\end{figure}

\begin{definition}
\label{D:SommeMurasugiIteree}
An annulus embedded in~$\Sph^3$ whose boundary is a positive Hopf link is called a~\emph{Hopf band}.

A surface~$\surf$ with boundary is an \emph{iterated Murasugi sum} if there exists Hopf bands~$\Hopf_1, \dots$, $\Hopf_n$, an increasing sequence of surfaces with boundary~$\Hopf_1=\surf_1 \subset \surf_2 \subset \cdots \subset \surf_n=\surf$, and two sequences of polygons~$\carre_1 \subset \surf_1, \ldots, \carre_{n-1} \subset \surf_{n-1}$ and $\carreb_2 \subset \Hopf_2, \ldots, \carreb_{n} \subset \Hopf_{n}$ such that, for every~$i$ between~$1$ and~$n-1$, the surface~$\surf_{i+1}$ is the Murasugi sum~$\Hopf_{i+1}\MurasommeB_{\carreb_{i+1}\sim\carre_{i}}{\surf_i}$.
The sequence~$\surf_1 \subset \surf_2 \subset \cdots \subset \surf_n=\surf$ is called a \emph{Murasugi realisation} of~$\surf$.
\end{definition}

All surfaces with boundary in~$\Sph^3$ are not iterated Murasugi sums. Indeed, the boundary of such a sum is a fibered link. This is therefore a very pecular situation.

\begin{figure}[htbp]
	\begin{center}
	\begin{picture}(300,350)(0,0)
	\put(0,0){\includegraphics[width=0.7\textwidth]{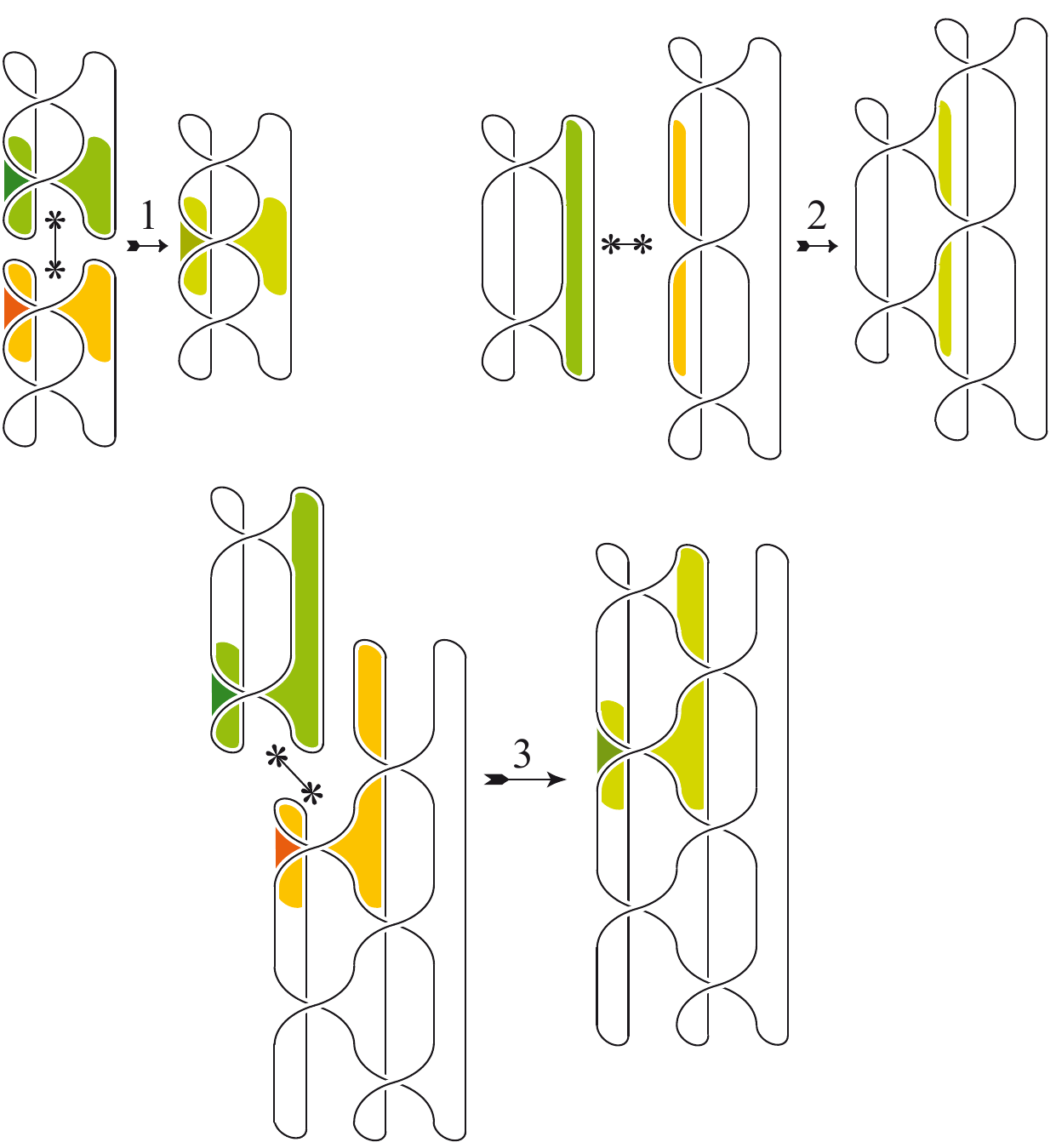}}
	\put(-15,235){$\Hopf_1$}
	\put(-15,300){$\Hopf_2$}
	\put(155,220){$\Hopf_3$}
	\put(47,160){$\Hopf_4$}
	\end{picture}
	\end{center}
	\caption{\small A realisation of the standard Seifert surface for the torus knot~$T(4,3)$. At each step, one takes the result of the previous step, and one glues on it a Hopf band along the colored polygon. The band~$\Hopf_1$ comes first, then~$\Hopf_2$... so that the Murasugi order (Definition~\ref{D:OrdreMurasugi}) associated to this realisation is $\Hopf_1\MuraOrdre \Hopf_2\MuraOrdre \Hopf_3\MuraOrdre \Hopf_4$.}
	\label{F:MurasugiIteree}
\end{figure}

Let $\surf$ be a surface admitting a Murasugi realisation~$\surf_1 \subset \surf_2 \subset \cdots \subset \surf_n=\surf$ along polygons $\Omega_1, \dots, \Omega_n$. If two consecutive polygons~$\Omega_i$ and $\Omega_{i+1}$ are disjoint in~$\Sigma_{i+1}$, then we can first glue~$\Hopf_{i+2}$ along~$\Omega_{i+1}$, and then~$\Hopf_{i+1}$ along~$\Omega_i$, and obtain the same surface~$\Sigma_{i+2}$ after these two steps. This means that we can change the order in which the bands~$\Hopf_{i+1}$ and $\Hopf_{i+2}$ are glued without changing the resulting surface.

Thus, for a fixed surface~$\Sigma$, there exists several possible orders for gluing the bands and realise~$\Sigma$. Nevertheless, some bands need to be glued before some others. For example if the gluing polygon~$\Omega_j$ intersects the band~$\Hopf_i$, then the band~$\Hopf_{j+1}$ has to be glued after~$\Hopf_i$.

\begin{definition}
\label{D:OrdreMurasugi}
Let~$\surf$ be an iterated Murasugi sum of~$n$ bands, denoted~$\Hopf_1, \cdots, \Hopf_n$. We say that the band~$\Hopf_i$ \emph{precedes} the band~$\Hopf_j$ in the \emph{Murasugi order associated to~$\surf$} if, for all possible realisations of~$\surf$, the band~$\Hopf_i$ is glued before~$\Hopf_j$. We then write~$\Hopf_i \MuraOrdre \Hopf_j$. 
\end{definition}

For every surface~$\Sigma$, the Murasugi order is a partial order on the set of Hopf bands whose union is~$\Sigma$.

\begin{proposition}
\label{T:DecompositionMonodromie}
Let~$K$ be an oriented link and~$\surf_K$ be a Seifert surface for~$K$ which is a Murasugi sum of Hopf bands~$\Hopf_1, \dots, \Hopf_n$. Let~$\gamma_{1}, \dots, \gamma_{n}$ be curves representing the cores of the bands~$\Hopf_1, \dots, \Hopf_n$. 
\begin{itemize}
\item[($i$)] The link~$K$ is fibered with fiber~$\surf_K$.
\item[($ii$)] Let~$\pi$ be a permutation of~$\{1, \cdots, n\}$ preserving Murasugi order, {\it i.e.}, such that $\Hopf_i \MuraOrdre \Hopf_j$ implies~$\pi(i) < \pi(j)$. Then the geometric monodromy of~$K$ is the composition of the positive Dehn twists~$\DTgeom_{\gamma_{\pi(n)}} \circ \cdots \circ \DTgeom_{\gamma_{\pi(1)}}$.
\end{itemize}
\end{proposition}

\begin{proof}
By definition, the sequence~$\Hopf_{\pi(1)}, \dots, \Hopf_{\pi(n)}$ induces a Murasugi realisation of~$\surf_K$. Since the monodromy of each Hopf band~$\Hopf_{\pi(i)}$ is the Dehn twist~$\DTgeom_{\gamma_{\pi(i)}}$, Theorem~\ref{T:sommefibre} implies that the link~$K$ is fibered, and that its monodromy is the composition~$\DTgeom_{\gamma_{\pi(n)}} \circ \cdots \circ \DTgeom_{\gamma_{\pi(1)}}$.
\end{proof}


\subsection{Standard surface for Lorenz knots}\label{S:SurfaceStandard}
\label{S:StandardSurface}

\begin{definition}\label{D:SurfaceStandard}
(See Figure~\ref{F:SurfaceStandard}.)
Let $D$ be a Young diagram and $K_D$ the associated Lorenz knot. The spanning surface~$\surfd$ for~$K_D$ obtained by gluing a disk beyond each strand and a ribbon at each crossing is called the~\emph{standard Seifert surface}.
\end{definition}

\begin{figure}[htbp]
	\begin{center}
	\includegraphics[width=0.4\textwidth]{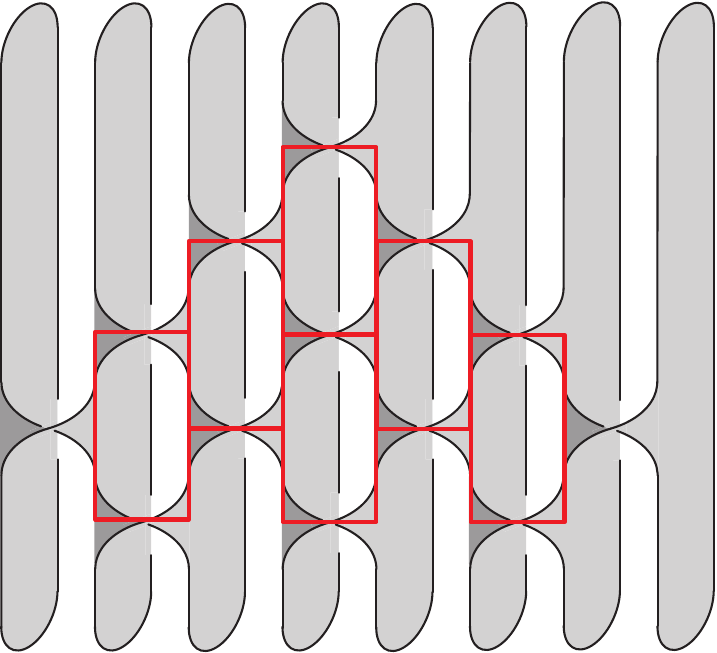}
	\end{center}
	\caption{\small The standard Seifert surface associated to the Young diagram~$[3,2,1]$, and the cores of the six Hopf bands that form a Murasugi realisation of the surface.}
	\label{F:SurfaceStandard}
\end{figure}

We now summarize the construction.

\begin{proposition}
\label{T:DecompositionMurasugi}
Let~$D$ be a Young diagram with~$n$ cells and~$K_D$ the associated Lorenz link.
\begin{itemize}
\item[($i$)] The standard Seifert surface~$\surfd$ is the iterated Murasugi sum of $n$ Hopf bands~$\Hopf_{i,j}$, each of them being associated with one of the $n$~cells $(i,j)$ de~$D$.
\item[($ii$)] The band~$\Hopf_{i_1,j_1}$ precedes~$\Hopf_{i_2,j_2}$ in the Murasugi order if and only if we have~$i_1\ge i_2$, $i_1+j_1 \ge i_2 + j_2,$ and $(i_1,j_1) \neq (i_2,j_2)$.
\item[($iii$)] For all cells~$(i,j)$ of~$D$, we choose a curve~$\gamma_{i,j}$ along the core of the band~$\Hopf_{i, j}$.
Then the family of classes $\{[\gamma_{i,j}]\}_{(i,j)\in D}$ forms a basis of~$H_1(\surfd; \Z)$ seen as a~$\Z$-module. 
The intersection number $\left\langle \gamma_{i_1,j_1} \, \big\vert \, \gamma_{i_2,j_2} \right\rangle$ is
\begin{equation*}
\begin{cases}
\label{Eq:NbIntersection}
+1 &\mbox{if $(i_2, j_2) = (i_1{+}1, j_1{+}1)$, $(i_1, j_1{-}2)$, or $(i_1{-}1, j_1{+}1)$}, \\
-1 &\mbox{if $(i_2, j_2) = (i_1{+}1, j_1{-}1)$, $(i_1, j_1{+}2)$, or $(i_1{-}1, j_1{-}1)$}, \\
0 &\mbox{otherwise.}
\end{cases}
\end{equation*}
\item[($iv$)] For every sequence $\Hopf_{i_1,j_1} \MuraOrdreLarge \cdots \MuraOrdreLarge \Hopf_{i_n,j_n}$ preserving the Murasugi order, the geometric monodromy of~$K$ is the product $\DTgeom_{\gamma_{i_1, j_1}} \circ \cdots \circ \DTgeom_{\gamma_{i_n, j_n}}$, and the homological monodromy is the product $\DT_{\gamma_{i_1, j_1}} \circ \cdots \circ \DT_{\gamma_{i_n, j_n}}$.
\end{itemize} 

\end{proposition}

\begin{proof}
For $(i)$ and $(ii)$, Figure~\ref{F:MurasugiIteree} shows how to glue~$n$ Hopf bands and obtain the surface~$\surfd$. We see that the band~$\Hopf_{i, j}$ is glued along a polygon included in the union of the three bands~$\Hopf_{i+1, j-1}$, $\Hopf_{i+1, j+1}$ and $\Hopf_{i, j+2}$. Therefore these bands need to be glued before adding~$\Hopf_{i, j}$. We obtain the result by induction.

$(iii)$ Given a cell~$(i, j)$ of~$D$, the homology of the band~$\Hopf_{i, j}$ is generated by the class~$[\gamma_{i, j}]$. Since the surface~$\surfd$ is the union of these Hopf bands, its homology is generated by~$\{[\gamma_{i,j}]\}_{(i,j)\in D}$. A computation of Euler characteristic of~$\surfd$ shows that these class form indeed a basis. We see on Figure~\ref{F:SurfaceStandard} that two curves~$\gamma_{i_1, j_1}, \gamma_{i_2, j_2}$ intersect only if the associated cells $(i_1, j_1)$ and $(i_2, j_2)$ of~$D$ are neighbours. The rule for signs is depicted on Figure~\ref{F:Intersection}.

$(iv)$ follows from~$(i)$ et Proposition~\ref{T:DecompositionMurasugi}.
\end{proof}

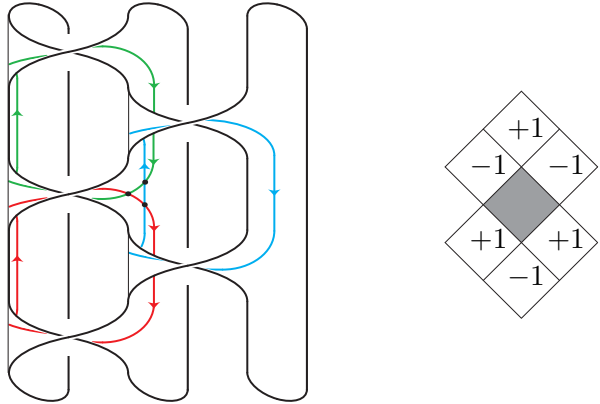
\begin{figure}[htbp]
	\begin{center}
	\begin{picture}(210,150)(0,0)
	\put(0,0){\includegraphics[width=0.5\textwidth]{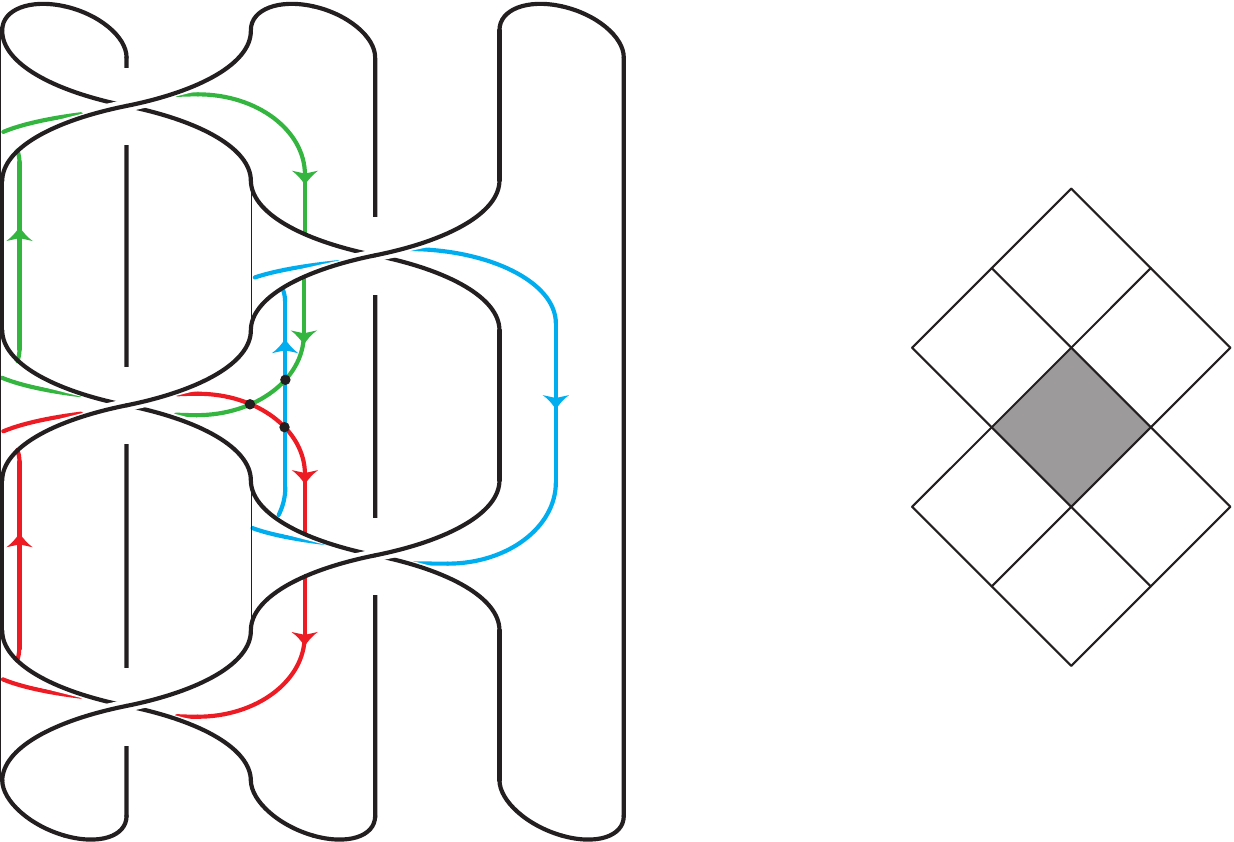}}
	\put(202,58){$+1$}
	\put(202,86){$-1$}
	\put(187,100){$+1$}
	\put(173,86){$-1$}
	\put(173,58){$+1$}
	\put(187,44){$-1$}
	\end{picture}
	\end{center}
	\caption{\small On the left, the curves $\gamma_{i,j}$, $\gamma_{i,j+2}$ and $\gamma_{i+1,j+1}$ on the surface~$\surfd$. Intersection points are dotted. On the right, values of the intersection between~$\gamma_{i, j}$ and curves associated with adjacent cells.}
	\label{F:Intersection}
\end{figure}

We now deduce the combinatorial form of the monodromy that we will rely on.

\begin{proposition}[see Figure~\ref{F:OrdreCompositionSimple}]
\label{T:FormuleMonodromie}
	Let~$D$ be a Young diagram and~$K$ the associated Lorenz knot. Then the homological monodromy associated to the standard Seifert surface is the composition 
	\[\prod_{\col=\col_r}^{\col_l} \prod_{j=\bbot_\col}^{\ttop_\col} \DT_{\gamma_{\col,j}}.\]
\end{proposition}

\begin{proof}
By Proposition~\ref{T:DecompositionMurasugi}$(ii)$, we have $\Hopf_{c_r,c_r} \MuraOrdreLarge \dots \MuraOrdreLarge \Hopf_{c,b_c} \MuraOrdreLarge \Hopf_{c, b_c-2}, \MuraOrdreLarge \dots \MuraOrdreLarge \Hopf_{c,t_c} \MuraOrdreLarge \Hopf_{c-1,b_{c-1}} \ldots$ The result then follows from Proposition~\ref{T:DecompositionMurasugi}$(iv)$.
\end{proof}

\begin{figure}[htbp]
	\begin{center}
	\includegraphics[width=0.4\textwidth]{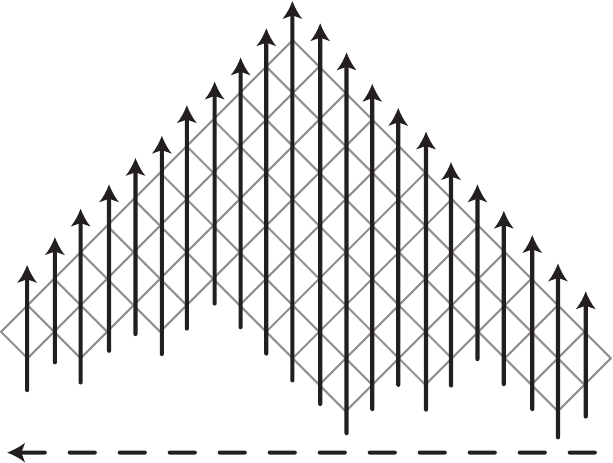}
	\end{center}
	\caption{\small A Murasugi order for the monodromy of a Lorenz link: we perform Dehn twists from right to left, and, in each column, from bottom to~top.}
	\label{F:OrdreCompositionSimple}
\end{figure}


\section{Combinatorics of the monodromy}
\label{S:Combinatorics}

Starting from a Lorenz knot~$K$, we obtained in Section~\ref{S:Preliminaries} a presentation for the monodromy~$h$ of~$K$ as a product of transvections. In this section, we analyze the image of particular cycles of the fiber of~$K$ under~$h$. Our goal is to find a basis of~$H_1(\Sigma; \Z)$ that splits into two families~$B_1, B_2$ so that the image under~$h$ of a cycle of~$B_1$ is another single cycle of~$B_1$ or $B_2$, and that the iterated images under~$h$ of a cycle of~$B_2$ stay in~$B_1$ for a number of steps with a uniform lower bound. We shall see in Section~\ref{S:Radius} that the existence of such a basis implies that the $\ell^1 $-norm of a cycle cannot grow too fast when the monodromy is iterated. 

We proceed in two steps. In Section~\ref{S:FirstPass}, we develop a first, relatively simple combinatorial analysis based on the standard Murasugi decomposition, and explain why it fails to provide a convenient basis. In Section~\ref{S:MixedSurface}, we introduce a new, more suitable Murasugi decomposition. Finally, in Section~\ref{S:SecondPass}, we complete the analysis for the latter decomposition and exhibit the expected basis.

\subsection{A first attempt}
\label{S:FirstPass}

From now on, we fix a Young diagram~$D$. We call~$K_D$ the associated Lorenz link, $\surfd$ the associated standard spanning surface for~$K_D$, viewed as an iterated Murasugi sum of~$n$ positive Hopf bands~$\Hopf_{i, j}$. For every cell $(i, j)$ in~$D$, we fix a curve~$\gamma_{i, j}$ that is the core of the Hopf band~$\Hopf_{i, j}$ embedded into~$\surfd$. By Proposition~\ref{T:DecompositionMurasugi}$(iii)$, the classes~$\{[\gamma_{i,j}]\}_{(i,j)\in D}$ form a basis of the group~$H_1(\surfd; \Z)$. We write $h_D$ for the homological monodromy associated with~$\surfd$, {\it i.e.}, the endomorphism of~$H_1(\surfd; \Z)$ induced by the geometrical monodromy. In order to improve readability, we write~$\cy{i, j}$ for the cycle~$[\gamma_{i, j}]$ and~$\DT\cy{i, j}$ for the transvection of~$H_1(\surfd; \Z)$ induced by a positive Dehn twist along the curve~$\gamma_{i, j}$.
We adopt the convention that, if $(i, j)$ are the coordinates of no cell of the diagram~$D$, then the curve~$\cy{i, j}$ is empty, and the twist~$\DT\cy{i, j}$ is the identity map on~$\surfd$.

\begin{lemma}
\label{T:DehnTwistLorenz}
Let~$\gamma$ be a curve on~$\surfd$. Suppose that its homology class admits the decomposition $[\gamma] = \displaystyle\sum_{k,l} x_{k,l} \cy{k,l}$. Then for all cells~$(i, j)$ of~$D$, we have
\begin{eqnarray*}
\DT{\cy{i,j}}\left([\gamma]\right) &=& [\gamma] + \left\langle \gamma \, \big\vert \, \cy{i,j} \right\rangle \cy{i,j} \\
& = & [\gamma] + \left(- x_{i{+}1,j{+}1} + x_{i,j{+}2} - x_{i{-}1,j{+}1} + x_{i{-}1,j{-}1} - x_{i,j{-}2} + x_{i{+}1,j{-}1}\right) \cy{i,j}. 
\end{eqnarray*}
\end{lemma}

\begin{proof}
The first equality comes from the definition of Dehn twists. The second one comes from the intersection numbers as computed in Proposition~\ref{T:DecompositionMurasugi}$(iii)$.
\end{proof}

For most cells in the diagram~$D$, the action of the monodromy~$h_D$ on the associated cycle is simple: it is sent on the cycle associated to an adjacent cell. The cells thus sent on adjcent cells are those that have a adjacent cell in SE-position.

\begin{definition}
\label{D:CelluleInterieure}
A cell with coordinates~$(i, j)$ in~$D$ is called~\emph{internal} if $D$ contains a cell with coordinates~$(i+1, j+1)$. It is called~\emph{external} otherwise.
\end{definition}

\begin{lemma}
\label{T:ImageInterieure}
For every~$(i,j)$ that refers to an internal cell of~$D$, we have
\begin{equation*}
h_D(\cy{i,j}) = \cy{i{+}1,j{+}1}.
\end{equation*}
\end{lemma}

\begin{proof}
Using the decomposition of~$h_D$ as a product of Dehn twists given by Proposition~\ref{T:FormuleMonodromie}, we see that the image of the cycle~$\cy{i, j}$ is given by
\begin{eqnarray}
\label{Eq:Interne1}
	\cy{i,j} & \donnepar{\cy{\col_r, \col_r}} & \cy{i,j} \donnepar{\cy{\col_r-1, \bbot_{\col_r-1}}} \cdots \donnepar{\cy{i{+}1, j{+}2}} \cy{i, j} \\
\label{Eq:Interne2}
	& \donnepar{\cy{i{+}1, j{+}1}} & \cy{i,j} {+} \cy{i{+}1, j{+}1} \\
	& \donnepar{\cy{i{+}1, j{-}1}} & (\cy{i,j} - \cy{i{+}1, j{-}1}) + (\cy{i{+}1, j{+}1} + \cy{i{+}1, j{-}1})\\
	& & = \cy{i,j} + \cy{i{+}1, j{+}1}\\
	& \donnepar{\cy{i{+}1, j{-}3}} & \cy{i,j} + \cy{i{+}1, j{+}1} \donnepar{\cy{i{+}1, j{-}5}} \dots \donnepar{\cy{i, j{+}4}} \cy{i,j} + \cy{i{+}1, j{+}1} \\
	& \donnepar{\cy{i, j{+}2}} & (\cy{i,j} - \cy{i, j{+}2}) + (\cy{i{+}1, j{+}1} + \cy{i, j{+}2}) \\
	& & =  \cy{i,j} + \cy{i{+}1, j{+}1}\\
	& \donnepar{\cy{i, j}} & \cy{i,j} + (\cy{i{+}1, j{+}1} - \cy{i,j}) = \cy{i{+}1, j{+}1}
\\
\label{Eq:Interne7}
	& \donnepar{\cy{i, j{-}2}} & \cy{i{+}1, j{+}1} \donnepar{\cy{i, j{-}4}} \cdots \donnepar{\cy{\col_l, \vert\col_l\vert}} \cy{i{+}1, j{+}1}.
\end{eqnarray}
Relation~\eqref{Eq:Interne1} comes from the fact that, for $i' > i{+}1$ and for all $j'$, and for $i' = i{+}1$ and $j' > j{+}1$, the intersection number $\langle \cy{i,j} \, \vert \, \cy{i',j'} \rangle$ is zero.
In the same way, for $i' < i$ and for all $j'$, and for $i' = i$ and $j' < j$, we have $\langle \cy{i{+}1,j{+}1} \, \vert \, \cy{i',j'} \rangle = 0$, implying~\eqref{Eq:Interne7}. 
Relation~\eqref{Eq:Interne2} follows from the equality $\left\langle \cy{i,j} \, \big\vert \, \cy{i{+}1,j{+}1} \right\rangle = 1$ stated in Proposition~\ref{T:DecompositionMurasugi}$(iii)$ and from Lemma~\ref{T:DehnTwistLorenz}. 
The other relations follow from similar observations.
\end{proof}

\begin{figure}[htbp]
	\begin{center}
	\includegraphics[width=0.4\textwidth]{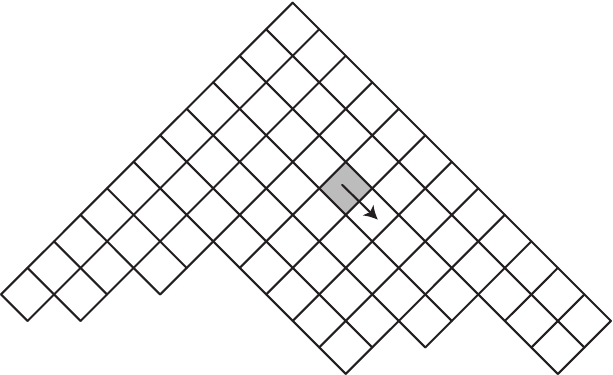}
	\end{center}
	\caption{\small The image of a cycle associated to an internal cell under the monodromy~$h_D$.}
	\label{F:MonodromieInterne}
\end{figure}

Let us turn to external cells, {\it i.e.} cells~$(i,j)$ such that $(i{+}1,j{+}1)$ is not a cell of~$D$.

\begin{definition}
\label{D:Rectangle}
Let $(i_1, j_1)$ and $(i_2, j_2)$ be two cells of diagram~$D$ satisfying $i_1\le i_2$ and $i_1 {+} j_1\ge i_2 {+} j_2$---geometrically this means that the cell~$(i_2, j_2)$ lies in the NNE-octant with respect to the cell~$(i_1,j_1)$. Then the \emph{rectangle}~$\rect{i_1,j_1}{i_2,j_2}$ is defined as the set of cells
\begin{equation*}
\left\{ ~{}~ (k,l)\in D ~{}~ \big\vert ~{}~ i_1\le k \le i_2 ~{}~ \mathrm{and} ~{}~ i_1+j_1 \ge k+l \ge i_2+j_2 ~{}~ \right\}.
\end{equation*}

In this case, the cells~$(i_1,j_1)$ and~$(i_2,j_2)$ are called the SW- and NE-corners of the rectangle respectively, and are denoted~$\mathrm{SW}(\rect{i_1,j_1}{i_2,j_2})$ and~$\mathrm{NE}(\rect{i_1,j_1}{i_2,j_2})$. 
\end{definition}


\begin{definition}
\label{D:EnsemblesAccesibles}
(See Figure~\ref{F:MonodromieExterne}.) Let $(i,j)$ be an external cell of the Young diagram~$D$. For $m=1, 2, \cdots$, we recursively define the \emph{accessible rectangle} $\acces_m(i,j)$ as follows\:

\noindent $(i)$ $\acces_1(i,j) = \rect{i+1,j-1}{i+1,\vert i+1\vert} = \left\{(i+1, j-1), (i+1, j-3), \cdots, (i+1, \vert i+1 \vert)\right\}$;

\noindent $(ii)$ if $\mathrm{NW}(\acces_m(i,j))+(-1,-1)$ is a cell of~$D$, then $\acces_{m+1}(i,j)$ is the rectangle whose SE-corner is the cell~$\mathrm{NW}(\acces_m(i,j))+(-1,-1)$, and whose NE- and SW-corners are on the boundary of the diagram~$D$ (this means that the cells $\mathrm{NE}(\acces_{m+1}(i,j))+(0,-2)$ and~$\mathrm{SW}(\acces_{m+1}(i,j))+(-1,1)$ are not in the diagram); else the construction stops and the rectangle~$\acces_{m'}(i,j)$ is empty for all~$m' > m$.

\vspace{-.6cm}

\end{definition}

\begin{figure}[htbp]
	\begin{center}
	\includegraphics[width=0.7\textwidth]{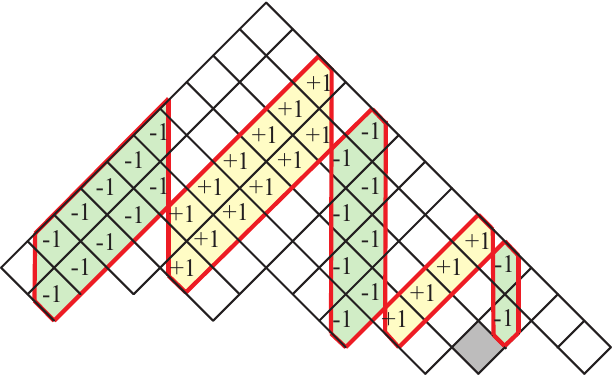}
	\end{center}
	\caption{\small The image of a cycle associated with an external cell under the monodromy.}
	\label{F:MonodromieExterne}
\end{figure}

Note that Definition~\ref{D:EnsemblesAccesibles} implies that, for every column, either no cell of the column lies in an accessible retangle, or some do, in which case they are adjacent, {\it i.e.}, of the form~$(c, t), (c, t+2), $ $\dots, (c, b)$, and they all belong to the same accessible rectangle.

\begin{lemma}
\label{T:ImageExterne}
Let $D$ be a Young diagram, $K_D$ be the associated Lorenz link, and $h_D$ be the associated monodromy. Assume that $(i,j)$ is an external of~$D$. Then we have
\[ h_D(\cy{i,j}) = \sum_{m\ge 1} \sum_{(k,l)\in\acces_m(i,j)} (-1)^m \cy{k,l}. \]
\end{lemma}

\begin{proof}
For every column~$\col$ of the diagram~$D$, we introduce a truncated product~$h_D^{\col}$ by
\begin{equation*}
h_D^{\col} = \prod_{k=\col}^{\col_r} \prod_{l=\ttop_k}^{\bbot_k} \tau{\cy{k,l}};
\end{equation*}
remember that $\col_r$ refers to the rightmost column of~$D$, and $b_k$ and $t_k$ denote the bottom and top cells of the column~$k$. For every~$\col$, by definition we have $h_D^{\col} =  \prod_{l=\ttop_\col}^{\bbot_\col} \tau{\cy{\col,l}}\circ h_D^{\col-1}$, and Lemma~\ref{T:FormuleMonodromie} implies $h_D = h_D^{\col_l}$. We will then evaluate each of terms~$h_D^{\col}(\cy{i, j})$ one after the other, for $\col$ going down from~$\col_r$ to~$\col_l$.

First suppose~$\col>i+1$. Then for every~$k\ge\col$ and for every~$l$, the intersection number~$\left\langle \cy{i,j} \, \big\vert \, \cy{k,l} \right\rangle$ is zero. We simply get~$h_D^{\col}(\cy{i,j}) = \cy{i,j}$.

Now suppose $c=i+1$. Owing to the decomposition~$h_D^{\col} =  \prod_{l=\ttop_\col}^{\bbot_\col} \tau{\cy{\col,l}}\circ h_D^{\col-1}$, we find
\begin{eqnarray*}
	\cy{i,j} &\donnepar{\cy{i{+}1, j{-}1}}& \cy{i,j} - \cy{i{+}1,j{-}1} \donnepar{\cy{i{+}1, j{-}3}} \cy{i, j} - (\cy{i{+}1,j{-}1} + \cy{i{+}1,j{-}3}) 
	\\
	&\donnepar{\cy{i{+}1, j{-}5}}& \cy{i, j} - \cy{i{+}1, j{-}1} - (\cy{i{+}1, j{-}3} + \cy{i{+}1, j{-}3}) 
	\donnepar{\cy{i{+}1, j{-}7}} \cdots 
	\\
	&\donnepar{\cy{i{+}1, \vert i \vert {+}1}}& \cy{i, j} - \cy{i{+}1, j{-}1} - \cy{i{+}1, j{-}3} - \cdots - \cy{i{+}1,\vert i \vert {+}1},
\end{eqnarray*}
where $h_D^{i+1}(\cy{i,j}) = \cy{i,j} + \sum_{(k,l) \in \acces_1(i,j)} - \cy{k,l}$.	 

Let us turn to the case $c=i$. We obtain similarly
\begin{eqnarray*}
	&&\cy{i, j} - \cy{i{+}1,j{-}1} - \cy{i{+}1,j{-}3} - \cdots - \cy{i{+}1,\vert i \vert {+}1} 
	\\
	&\donnepar{\cy{i, j}}& \cy{i, j} - (\cy{i{+}1,j{-}1} + \cy{i,j}) - \cy{i{+}1,j{-}3} - \cdots - \cy{i{+}1,\vert i \vert {+}1} 
	\\
	&=& - \cy{i{+}1,j{-}1} - \cy{i{+}1,j{-}3} - \cdots - \cy{i{+}1,\vert i \vert {+}1} 
	\\
	&\donnepar{\cy{i, j{-}2}}& - (\cy{i{+}1,j{-}1} - \cy{i,j{-}2}) - (\cy{i{+}1,j{-}3} + \cy{i,j{-}2}) - \cdots - \cy{i{+}1,\vert i \vert {+}1} 
	\\
	&=& - \cy{i{+}1,j{-}1} - \cy{i{+}1,j{-}3} - \cdots - \cy{i{+}1,\vert i \vert {+}1} \donnepar{\cy{i, j{-}4}} \cdots
	\\
	&\donnepar{\cy{i, \vert i \vert {+}2}}& - \cy{i{+}1,j{-}1} - \cy{i{+}1, j{-}3} - \cdots - \cy{i, \vert i \vert {+} 2} - (\cy{i{+}1, \vert i \vert {+}1} + \cy{i, \vert i \vert {+}2}) 
	\\
	&=& - \cy{i{+}1,\vert i \vert {+}1} - \cy{i{+}1, j{-}1} - \cy{i{+}1, j{-}3} - \cdots - \cy{i{+}1, \vert i \vert {+}1} 
	\\
	&\donnepar{\cy{i, \vert i \vert}}& - \cy{i{+}1, j{-}1} - \cdots - \cy{i{+}1,\vert i \vert {+}1} - (\cy{i{+}1,\vert i \vert {+}1} - \cy{i, \vert i \vert})
	= \cy{i, \vert i \vert} + \sum_{(k,l) \in \acces_1(i,j)} -\cy{k,l}
\end{eqnarray*}
hence $h_D^{i}(\cy{i,j}) = \cy{i,\vert i \vert} + \sum_{(k,l) \in \acces_1(i,j)} - \cy{k,l}$. Observe that the latter expression can be written $\sum_{m\ge 1} \sum_{(k,l)\in\acces_m(i,j), k\ge i} (-1)^m \cy{k,l}$.

We now look at the case~$c<i$. On the shape of the last expression, let us show that for every $c<i$ we have
\begin{equation}
\label{Eq:hD}
h_D^{\col}(\cy{i,j}) = \sum_{m\ge 1} \sum_{(k,l)\in\acces_K^m(i,j), k> \col} (-1)^m \cy{k,l}. 
\end{equation}
We use a induction with $\col$ going down from~$i{-}1$ to~$\col_l$. There are two cases.

\noindent{\it Case 1. There exists an index~$\rec$ so that at least one cell of the $\col{+}1$st column lies in the rectangle~$\acces_{\rec}(i,j)$.} 
Let $(\col{+}1, \toprec), (\col{+}1, \toprec{+}2), \cdots, (\col{+}1, \botrec)$ denote the cells of this column lying in~$\acces_{\rec}(i,j)$.  When transforming~$h_D^{\col{+}1}(\cy{i,j})$ into~$h_D^{\col}(\cy{i,j})$, we perform Dehn twists along curves associated to the $\col$th column. Since the only cycles in $h_D^{\col{+}1}(\cy{i,j})$ having non-zero intersection with curves associated to the $\col$th column are of the form~$[\col{+}1, l]$, these are the only cycles that are modified when transforming~$h_D^{\col{+}1}(\cy{i,j})$ into~$h_D^{\col}(\cy{i,j})$. By induction hypothesis, we have
\begin{equation}
\label{Eq:Decomposition}
	h_D^{\col+1}(\cy{i,j}) = \sum_{m\ge 1} \sum_{(k,l)\in\acces_K^m(i,j), k> \col+1} (-1)^m \cy{k,l} + 
		 \sum_{l=\toprec}^{\botrec} (-1)^{\rec} \cy{\col{+}1,l}.
\end{equation}
Call~$S^c$ the first term in the right-hand side of~\eqref{Eq:Decomposition}. We just noted that Dehn twists along curves of the~$\col$th column do not modify~$S^c$. Therefore we only consider the action of the composition~$\prod_{l=\ttop_\col}^{\bbot_\col} \tau{\cy{\col,l}}$ on the cycle~$\sum_{l=\toprec}^{\botrec} \cy{\col{+}1,l}$. In order to evaluate the result, we again separate two cases, depending on whether the $c{+}1$st column contains the west-border of a rectangle or not.

\noindent {\it Subcase 1.1. The cells $(\col{+}1, \toprec), (\col{+}1, \toprec{+}2), \cdots, (\col{+}1, \botrec)$ are not on the west side of the rectangle~$\acces_{\rec}(i,j)$.} We apply the twists associated to the cells of the $c$th column, and get
\begin{multline*}
	\cy{\col{+}1,\botrec} + \cy{\col{+}1,\botrec{-}2} + \cdots + \cy{\col{+}1,\toprec}
	\donnepar{\cy{\col, \botcol}} \cy{\col{+}1,\botrec} + \cy{\col{+}1,\botrec{-}2} + \cdots + \cy{\col{+}1,\toprec}
	\donnepar{\cy{\col, \botcol{-}2}} \cdots 
	\\
	\donnepar{\cy{\col, \botrec{+}1}} (\cy{\col{+}1,\botrec} + \cy{\col,\botrec{+}1}) + \cy{\col{+}1,\botrec{-}2} + \cdots + \cy{\col{+}1,\toprec}	
	= \cy{\col{+}1,\botrec} + \cy{\col{+}1,\botrec{-}2} + \cdots + \cy{\col,\botrec{+}1} 
	\\
	\donnepar{\cy{\col, \botrec{-}1}} (\cy{\col{+}1,\botrec} - \cy{\col,\botrec{-}1}) + (\cy{\col{+}1,\botrec{-}2} + \cy{\col,\botrec{-}1})
	+ \cdots + \cy{\col{+}1,\toprec} + (\cy{\col{+}1,\botrec{+}1} + \cy{\col,\botrec{-}1}) 
	\\
	= \cy{\col{+}1,\botrec} + \cy{\col{+}1,\botrec{-}2} + \cdots + \cy{\col{+}1,\toprec}
	+ \cy{\col,\botrec{+}1} + \cy{\col,\botrec{-}1}  
	\\
	\donnepar{\cy{\col, \botrec{-}3}} \cy{\col{+}1,\botrec} + (\cy{\col{+}1,\botrec{-}2} - \cy{\col,\botrec{-}3}) 
	+ (\cy{\col{+}1,\botrec{-}4} + \cy{\col,\botrec{-}3}) + \cdots + \cy{\col{+}1,\toprec} + \cy{\col,\botrec{+}1} 
	\\
	+ (\cy{\col,\botrec{-}1} + \cy{\col,\botrec{-}3}) 
	= \cy{\col{+}1,\botrec} + \cy{\col{+}1,\botrec{-}2} + \cdots + \cy{\col{+}1,\toprec} 
	+ \cy{\col,\botrec{+}1} + \cy{\col,\botrec{-}1} + \cy{\col,\botrec{-}3} 
	\\
	\donnepar{\cy{\col, \botrec{-}5}} \cdots 
	\donnepar{\cy{\col, \toprec{-}1}} \cy{\col{+}1,\botrec} + \cy{\col{+}1,\botrec{-}2} + \cdots + \cy{\col{+}1,\toprec} 
	+ \cy{\col,\botrec{+}1} + \cy{\col,\botrec{-}1} + \cdots + \cy{\col,\toprec{+}1} 
	\\	 
	\donnepar{\cy{\col, \toprec-1}} \cy{\col+1,\botrec} + \cy{\col+1,\botrec-2} + \cdots + (\cy{\col+1,\toprec} - \cy{\col,\toprec-1})
	+ \cy{\col,\botrec+1} + \cy{\col,\botrec-1} + \cdots 
	\\ 
	+ (\cy{\col,\toprec+1} + \cy{\col,\toprec-1}) 
	= \cy{\col+1,\botrec} + \cy{\col+1,\botrec-2} + \cdots + \cy{\col+1,\toprec} + \cy{\col,\botrec+1} + \cdots + \cy{\col,\toprec+1} 
	\\
	\donnepar{\cy{\col, \toprec-3}} \cdots
	\donnepar{\cy{\col, \vert \col \vert}} \cy{\col+1,\botrec} + \cy{\col+1,\botrec-2} + \cdots + \cy{\col+1,\toprec}
	+ \cy{\col,\botrec+1} + \cy{\col,\botrec-1} + \cdots + \cy{\col,\toprec+1}.
\end{multline*}

By adding the unchanged term~$S^c$, we get~\eqref{Eq:hD}, as expected.

\noindent{\it Subcase 1.2. The cells~$(\col{+}1, \toprec), (\col{+}1, \toprec{+}2), \cdots, (\col{+}1, \botrec)$ are on the west side of the rectangle~$\acces_{\rec}(i,j)$.} Then the cell~$(c, \botcol)$ lies in the rectangle~$\acces_{r{-}1}(i, j)$, and by the definition of~$\acces_{\rec}(i,j)$ the diagram~$D$ contains no cell in position~$(c{+}1,\botcol{+}1)$, implying $\botrec > \botcol$.
The cells of the $\col$th column that lie in~$\acces_{\rec{+}1}(i,j)$ are of the form~$(\col, \botrecc), (\col, \botrecc + 2), \cdots, (\col, \toprecc)$. Moreover by the definition of~$\acces_{\rec{+}1}(i,j)$ we have $\botrecc=\toprec+1$ and $\toprecc = \vert \col\vert$.
We deduce
\begin{multline*}
	\cy{\col{+}1,\botrec} + \cy{\col{+}1,\botrec{-}2} + \cdots + \cy{\col{+}1,\toprec}
	\donnepar{\cy{\col,\botcol}} \cy{\col{+}1,\botrec} + \cy{\col{+}1,\botrec{-}2} + \cdots + (\cy{\col{+}1,\botcol+1} - \cy{\col,\botcol})\\
	+ (\cy{\col{+}1,\botcol{-}1} + \cy{\col,\botcol}) + \cdots + \cy{\col{+}1,\toprec}
	= \cy{\col{+}1,\botrec} + \cy{\col{+}1,\botrec{-}2} + \cdots + \cy{\col{+}1,\toprec} 
	\donnepar{\cy{\col,\botcol{-}2}} \cdots \\
	\donnepar{\cy{\col,\botrecc{+}2}} \cy{\col{+}1,\botrec} + \cy{\col{+}1,\botrec{-}2} + \cdots + (\cy{\col{+}1,\toprec{+}2}
	- \cy{\col,\botrecc{+}2}) + (\cy{\col{+}1,\toprec} + \cy{\col,\botrecc{+}2})\\
	= \cy{\col{+}1,\botrec} + \cy{\col{+}1,\botrec{-}2} + \cdots + \cy{\col{+}1,\toprec} 
	\donnepar{\cy{\col,\botrecc}} \cy{\col{+}1,\botrec} + \cy{\col{+}1,\botrec{-}2} + \cdots + (\cy{\col{+}1,\toprec} - \cy{\col, \botrecc}) \\
	= \cy{\col{+}1,\botrec} + \cy{\col{+}1,\botrec-2} + \cdots + \cy{\col{+}1,\toprec} - \cy{\col, \botrecc} \\
	\donnepar{\cy{\col,\botrecc{-}2}} \cy{\col+1,\botrec} + \cy{\col{+}1,\botrec{-}2} + \cdots + \cy{\col{+}1,\toprec} 
	- (\cy{\col, \botrecc} + \cy{\col, \botrecc{-}2}) 
	\donnepar{\cy{\col,\botrecc{-}4}} \cdots \\
	\donnepar{\cy{\col,\topcol}} \cy{\col{+}1,\botrec} + \cy{\col{+}1,\botrec{-}2} + \cdots + \cy{\col{+}1,\toprec} - \cy{\col, \botrecc}
	 - \cy{\col, \botrecc{-}2} - \cdots - (\cy{\col, \topcol{+}2} + \cy{\col, \topcol}) \\
	= \cy{\col{+}1,\botrec} + \cy{\col{+}1,\botrec{-}2} + \cdots + \cy{\col{+}1,\toprec}
	- \cy{\col, \botrecc} - \cy{\col, \botrecc{-}2} - \cdots - \cy{\col, \topcol},
\end{multline*}
and once again, by adding the unchanged term~$S^c$, we get~\eqref{Eq:hD}.


\noindent{\it Case 2. No cell of the $\col{+}1$st column lies in an accessible rectangle.} Then the Dehn twists associated to the cells of the~$\col$th column do not modify the cycles of~$h^{\col{+}1}_D(\cy{i,j})$, implying $h^{\col}_D(\cy{i,j}) = h^{\col{+}1}_D(\cy{i,j})$.

The induction is complete. Since the expression of~$h_D^{\col_l}(\cy{i,j})$ coincides with the desired expression for~$h_D(\cy{i, j})$, the proof is complete.
\end{proof}

\subsection{Other spanning surfaces for positive links}
\label{S:MixedSurface}

Our strategy for finding bounds on the eigenvalues of the homological monodromy~$h_D$ associated to a Young diagram~$D$ is to bound the growth rate of all elements of~$H_1(\surfd; \Z)$ when iterating the endomorphism~$h_D$. Using the combinatorial information given by Lemmas~\ref{T:ImageInterieure} and~\ref{T:ImageExterne}, one can devise the following plan. The $\ell^1$-norm of a cycle increases under $h_D$ only if it has non-zero coordinates corresponding to external cells, in which case the norm is multiplied by at most~$n$, where~$n$ is the number of cells of the diagram~$D$. Thus, if we could find a lower bound~$t_0$ for the time needed for the first external cell to appear in the iterates~$h_D^t(c)$, then we would deduce that the $\ell^1$-norm grows asymptotically like~$n^{t/{t_0}}$. This would imply that the moduli of the eigenvalues of~$h_D$ are lower than~$(\log n)/{t_0}$. Unfortunately, the information we have on the monodromy so far does not enable us to have such a lower bound on the ``time of first return in an external cell''.

The goal for the end of the section is to take advantage of the flexibility in the choice of the spanning surface---actually the choice in the presentation of the surface---to obtain another expression for the monodromy, and to let the strategy work.

Let $b$ be a braid and let~$K$ be its closure. The standard way of drawing~$K$ consists in connecting the top and bottom extremities of~$b$ with strands behind~$b$ (Figure~\ref{F:SurfaceStandard} again). However, we may as well connect with strands in front of~$b$, or even use a combination of back and front connections, without changing the isotopy class of~$K$. As displayed on Figure~\ref{F:Radiateur}, a spanning surface of~$K$ is associated with every such combination. This spanning surface is always an iterated Murasugi sum, but the Murasugi order of the Hopf bands depends on the choice of a front or a back connection for each strand, and so does the presentation of the monodromy of~$K$.

\begin{figure}[htbp]
	\begin{center}
	\includegraphics[width=0.5\textwidth]{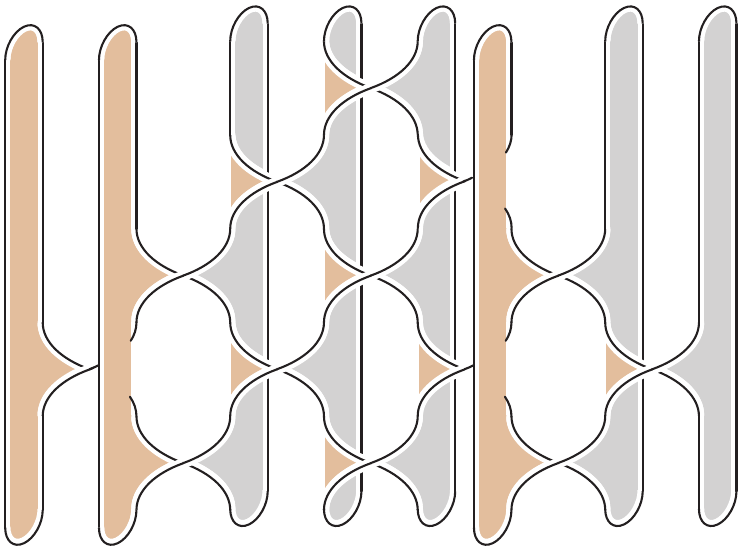}
	\end{center}
	\caption{\small The spanning surface associated to the diagram~$[3,2,1]$ and to the choice $\sigma=\{+,+,-,-,-,+,-,-\}$.}
	\label{F:Radiateur}
\end{figure} 

\begin{definition}
\label{D:Radiateur}
Let $b$ be a braid with~$s$ strands, and $\sigma$ be an element of~$\{+, -\}^s$. We define~$\hat b^\sigma$ to be the diagram obtained from~$b$ by connecting the top and bottom ends of the $i$th strand in front the braid~$b$ if the $i$th element of~$\sigma$ is~$+$, and beyond~$b$ if it is~$-$. We define~$\surf_{b}^\sigma$ to be the surface obtained by applying the Seifert algorithm to~$\hat b^\sigma$, {\it i.e.} by adding $s$ disks filling the connecting strands of~$\hat b^\sigma$, and connecting them to their neighbours with ribbons attached at each crossing (see Figure~\ref{F:Radiateur} for an example).
\end{definition}

The knot defined by~$\hat b^\sigma$ does not depend on~$\sigma$, since we can move strands from ahead to behind using isotopies. But there is no reason that these isotopies extend to the surfaces~$\surf_{b}^\sigma$. Nethertheless, because the knot~$K$ is fibered, it admits a unique spanning surface of minimal genus, and, therefore, all surfaces associated to various choices~$\sigma$ must be isotopic.

For every~$\sigma$, the surface~$\surf_{b}^{\sigma}$ is an iterated Murasugi sum of Hopf bands. While this surface is similar to~$\Sigma$, the combinatorics associated with~$\surf_b^\sigma$ is in general different from the one associated with~$\Sigma$. It turns out that the following choice enables us to realise our strategy for finding bounds on eigenvalues. 

\begin{definition}
\label{D:RadiateurLorenz}
(See Figure~\ref{F:RadiateurLorenz}.)
Let~$D$ be a Young diagram. Let $b$ denote the associated Lorenz braid, $s$ its number of strands, and $K$ its closure. Remember that we write~$\col_l$ ({\it resp.\ }$\col_r$) for the index of the left ({\it resp.\ }right) column of~$D$.
Let $\sigma_D$ denote the element~$ (+, +, \ldots, -, \ldots)$ of~$\{+, -\}^{s}$, with $\col_l{+}1$ $+$signs and $\col_r{+}1$ $-$signs. We define the \emph{mixed Seifert surface~$\seifbis$ of~$K$} to be the surface~$\surf_{b}^{\sigma_D}$. 
\end{definition}

\begin{figure}[htbp]
	\begin{center}
	\includegraphics[width=0.5\textwidth]{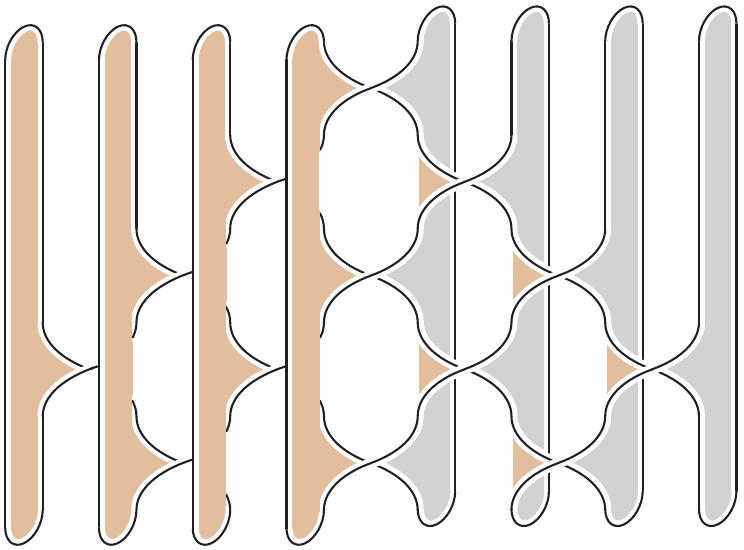}
	\end{center}
	\caption{\small The mixed Seifert surface associated to the diagram~$[3,2,1]$.}
	\label{F:RadiateurLorenz}
\end{figure} 

We can now derive an analog of Proposition~\ref{T:DecompositionMurasugi} for the mixed surface. In the sequel, it will be necessary to consider the inverse~$h_D^{-1}$ of the monodromy instead of~$h_D$---the advantage being that external cells will be replaced by central cells, whose images under~$h_D^{-1}$ are better controlled than the images of external cells under~$h_D$. 

\begin{proposition}
\label{T:DecompositionMurasugiBis}

Let $D$ be a Young diagram with $n$ cells. Let $K_D$ be the associated Lorenz knot and $\seifbis$ be associated mixed Seifert surface.
\begin{itemize}
\item[($i$)] The surface~$\seifbis$ is an iterated Murasugi sum of $n$ Hopf bands~$\wt\Hopf_{i,j}$, each of them associated to one of the $n$ cells of~$D$.
\item[($ii$)] The band~$\wt\Hopf_{i_1,j_1}$ is before~$\wt\Hopf_{i_2,j_2}$ in Murasugi order if and only if we have
\begin{eqnarray*}
 & & i_1, i_2 \ge 0, \quad i_1\ge i_2, \quad i_1+j_1 \ge i_2 + j_2, \quad \mbox{and} \quad (i_1,j_1) \neq (i_2,j_2) \\
 &\mbox{or}& i_1, i_2 \le 0, \quad \vert i_1\vert \ge \vert i_2\vert, \quad \vert i_1\vert + j_1 \ge \vert i_2\vert + j_2, \quad \mbox{and} \quad (i_1,j_1) \neq (i_2,j_2).
\end{eqnarray*}
\item[($iii$)] For each cell~$(i,j)$ of~$D$, choose a curve~$\wt\gamma_{i,j}$ that is the core of the annulus~$\wt\Hopf_{i, j}$. Then the cycles $\{[\wt\gamma_{i,j}]\}_{(i,j)\in D}$ form a basis of~$H_1(\seifbis; \Z)$. The intersection number $\left\langle \wt\gamma_{i_1,j_1} \, \big\vert \, \wt\gamma_{i_2,j_2} \right\rangle$ is
\begin{equation*}
\begin{cases}
\ +1 &\mbox{if $(i_2, j_2) = (i_1+1, j_1+1)$, $(i_1, j_1-2)$, or $(i_1-1, j_1+1)$}, \\
\ -1 &\mbox{if $(i_2, j_2) = (i_1+1, j_1-1)$, $(i_1, j_1+2)$, or $(i_1-1, j_1-1)$}, \\
\ \ 0 &\mbox{otherwise}.
\end{cases}
\end{equation*}
\item[($iv$)] Denote by $\wt\tau[i,j]$ the Dehn twist of The monodromy~$h_D$ of~$K$ satisfies 
\begin{eqnarray*}
h_D &=& \prod_{\col=\col_r}^{1} \prod_{j=\bbot_\col}^{\ttop_\col} \tau{\cybis{\col,j}} 
\circ  \prod_{\col=\col_l}^{-1} \prod_{j=\bbot_\col}^{\ttop_\col} \tau{\cybis{\col,j}} 
\circ \prod_{j=\bbot_0}^{\ttop_0} \tau{\cybis{0,j}},
\\
h_D^{-1} &=& \prod_{j=\ttop_0}^{\bbot_0} \tau{\cybis{0,j}}^{-1}
\circ  \prod_{\col=-1}^{\col_l} \prod_{j=\ttop_\col}^{\bbot_\col} \tau{\cybis{\col,j}}^{-1} 
\circ \prod_{\col=1}^{\col_r} \prod_{j=\ttop_\col}^{\bbot_\col} \tau{\cybis{\col,j}}^{-1}.
\end{eqnarray*}
\end{itemize}
\end{proposition}

\begin{proof}
The proof of $(i)$, $(ii)$, $(iii)$ is similar to the proof of their counterparts in Proposition~\ref{T:DecompositionMurasugiBis}. As for~$(iv)$, Dehn twists are performed in the order depicted in Figure~\ref{F:OrdreMonodromieBis}. It is compatible with the Murasugi order of~$(ii)$. The expression for~$h_D$ then follows from Proposition~\ref{T:DecompositionMonodromie}.
\end{proof}

\begin{figure}[htbp]
	\begin{center}
	\includegraphics[width=0.4\textwidth]{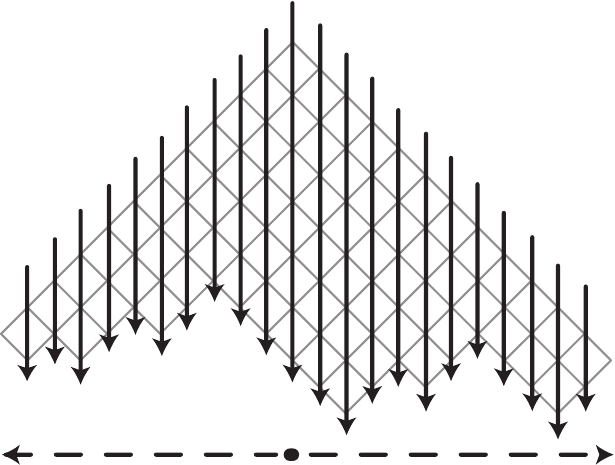}
	\end{center}
	\caption{\small The order of Dehn twists for (the inverse of) the monodromy; this order is compatible with the Murasugi order associated to the mixed Seifert surfaces.}
	\label{F:OrdreMonodromieBis}
\end{figure}


\subsection{Combinatorics of the monodromy: second attempt}
\label{S:SecondPass}

All results of Section~\ref{S:FirstPass} can now be restated in the context of mixed Seifert surface. The cells of the Young diagram can no longer be partitioned into internal and external cells, but, instead, we use the five types displayed on Figure~\ref{F:TypeCellules}. Hereafter we shall complete the computation for the inverse~$h_D^{-1}$ of the monodromy, which turns out to be (slightly) simpler that the computation of~$h_D$.

In this part, we fix a Young diagram~$D$ with $n$~cells. Let~$K$ be the associated Lorenz knot, and $\seifbis$ be the associated mixed Seifert surface, seen as an iterated Murasugi sum of~$n$ Hopf bands. Let~$\{[\gamma_{i, j}]\}_{(i, j)\in D}$ be a family of curves on~$\seifbis$, each of them being the core of one of the Hopf bands. We write $\cybis{i, j}$ for the class of~$\gamma_{i, j}$ in~$H_1(\seifbis; \Z)$.

The analog of Lemma~\ref{T:DehnTwistLorenz} is

\begin{lemma}
\label{T:DehnTwistLorenzBis}
Let $\gamma$ be a curve on~$\seifbis$. Suppose that its class~$\cybis\gamma$ in~$H_1(\seifbis; \Z)$ is equal to~$\displaystyle\sum_{k,l} x_{k,l} \cybis{k,l}$. Then for every cell~$(i,j)$ of~$D$, we have
\begin{eqnarray*}
\tau\cybis{i,j}^{-1}\left(\cybis{\gamma}\right) &=& \cybis{\gamma} - \left\langle \cybis{h} \, \big\vert \, \cybis{i,j} \right\rangle \cybis{i,j} \\
 &= & \cybis{\gamma} + \left( x_{i+1,j+1} - x_{i,j+2} + x_{i-1,j+1} - x_{i-1,j-1} + x_{i,j-2} - x_{i+1,j-1}\right) \cybis{i,j}. 
\end{eqnarray*}

\vspace{-.35cm}

\end{lemma}

The analog of internal cells -- the cells whose image under~$h_D^{-1}$ is an adjacent cell -- are the \emph{peripheral} cells. 

\begin{definition}
\label{D:CelluleCentrale}
A cell of~$D$ with coordinates~$(i,j)$ is called \emph{central} for $i=0$, \emph{right medial} for $i = 1$, \emph{left medial} for $i = -1$, \emph{right peripheral} for $i > 1$, and \emph{left peripheral} for $i < -1$.
\end{definition}

\begin{lemma}
\label{T:ImagePeripherique}
Assume that $(i,j)$ is a right ({\it resp.\ }left) peripheral cell of~$D$. Then we have
\begin{equation*}
h_D^{-1}(\cybis{i,j}) = \cybis{i{-}1,j{-}1} \quad ({\it resp.}~\cybis{i{+}1, j{-}1}).
\end{equation*}
\end{lemma}

The proof mimics the one of Lemma~\ref{T:ImageInterieure}.
 
\begin{figure}[htbp]
	\begin{center}
	\includegraphics[width=0.4\textwidth]{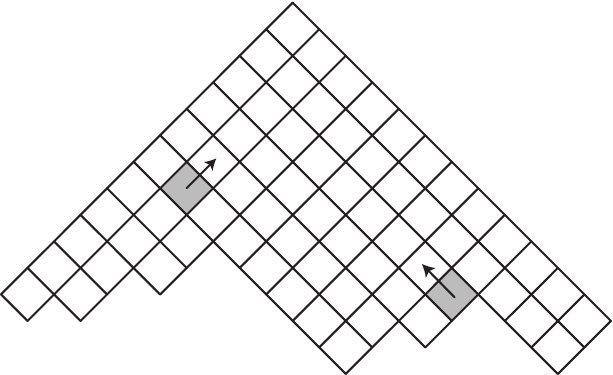}
	\end{center}
	\caption{\small Images of peripheral cells under the inverse of the monodromy.}
	\label{F:MonodromiePeripherique}
\end{figure}

Now, we are looking for an expression of the images under~$h_D^{-1}$ of central, right medial and left medial cells. Actually, rather than central cells, we look at another family of cycles, whose image is simpler.

\begin{definition}
\label{D:Equerre}
Let $(0,j)$ denote a central cell of~$D$, we call~\emph{try square} the set~$\equerre{j}$ of cells $\big\{ (0,j), (-1,j{-}1), (1,j{-}1), (-2, j{-}2), (2, j{-}2), \cdots (-j/2, j/2), (j/2,j/2) \big\}$.
\end{definition}

\begin{lemma}
\label{T:ImageEquerre}
Let $\equerre{j}$ be a try square of~$D$. Then we have
\begin{equation*}
h_D^{-1}\big(\sum_{(k,l)\in\equerre{j}}\cybis{k,l}\big) = \sum_{(k,l)\in\equerre{j{-}1}}\cybis{k,l}. 
\end{equation*}
\end{lemma}

\begin{proof}
Lemma~\ref{T:ImagePeripherique} describes the images of all cells of the try square~$\equerre{j}$ under~$h_D^{-1}$, except the cells~$(0,j), (-1, j{-}1)$ and~$(1, j{-}1)$. 
It is therefore sufficient to show the equality~$h_D^{-1}\big(\cybis{0,j}+\cybis{-1, j{-}1}+\cybis{1, j{-}1}\big) = \cybis{0,j{-}2}$. 
Using Proposition~\ref{T:DecompositionMurasugiBis}$(iv)$, and considering only the twists that modify the cycle we are considering, we obtain
\begin{eqnarray*}
& & \cybis{0,j} +\cybis{-1, j{-}1}+\cybis{1, j{-}1} \\
& \donnepar{\cybis{0,j{-}2}}& (\cybis{0,j}-\cybis{0,j{-}2})+(\cybis{-1, j{-}1}+\cybis{0,j{-}2})+(\cybis{1, j{-}1}+\cybis{0,j{-}2}) \\
& = & \cybis{0,j{-}2}+ \cybis{0,j}+(\cybis{-1, j{-}1} + \cybis{1, j{-}1} \\
& \donnepar{\cybis{0,j}}& (\cybis{0,j{-}2}+\cybis{0,j}) + \cybis{0,j} + \cybis{-1, j{-}1}-\cybis{0,j})+(\cybis{1, j{-}1}-\cybis{0,j}) \\
& = & \cybis{0,j{-}2} + \cybis{-1, j{-}1} + \cybis{1, j{-}1} \\
& \donnepar{\cybis{1,j{-}3}}& (\cybis{0,j{-}2} + \cybis{1,j{-}3}) + \cybis{-1, j{-}1} + (\cybis{1, j{-}1} - \cybis{1,j{-}3}) \\
& = & \cybis{0,j{-}2} + \cybis{-1, j{-}1} + \cybis{1, j{-}1} \\
& \donnepar{\cybis{1,j{-}1}}& (\cybis{0,j{-}2} - \cybis{1,j{-}3}) + \cybis{-1, j{-}1} + \cybis{1, j{-}1} = \cybis{0,j{-}2} + \cybis{-1, j{-}1} \\
& \donnepar{\cybis{-1,j{-}1}} & \cybis{0,j{-}2},
\end{eqnarray*}
as expected.
\end{proof}

Accessible rectangles also have an analog: accessible rays.

\begin{definition}
\label{D:Ray}
Let $(i, j)$ be the coordinates of a cell of the Young diagram. 
Then the \emph{left ray}~$\lray{i,j}$ is defined as the set of cells~$\left\{(k,l) ~\big\vert~ k \le i ~\mathrm{ and }~ k + l = i + j \right\}$, the \emph{right ray}~$\rray{i,j}$ is defined as the set~$\left\{ (k,l) ~\big\vert~ k \ge i ~\mathrm{ and }~ k - l = i - j \right\}$, and the \emph{vertical ray}~$\vray{i,j}$ as the set $\left\{ (k,l) ~\big\vert ~ l \ge j ~\mathrm{ and }~ k = i \right\}$.
The top and bottom cells of a ray are defined in the obvious way, and are denoted~$\topcell{\lray{i,j}}$ and~$\botcell{\lray{i,j}}$ respectively.
\end{definition}

\begin{lemma}
\label{T:ImageCentreDroit}
Let $(1,j)$ be a right medial cell of the diagram~$D$. Then we recursively define the accessible sets~$\acces_m(1,j)$ as follows:
\begin{itemize}
\item[($i$)] the set~$\acces_0(1,j)$ is the ray~$\lray{0,j-1}$;
\item[($ii$)] the set~$\acces_1(1,j)$ is the ray~$\vray{\topcell{\acces_K^0(1,j)}+(-1,1)}$;
\item[($iii$)] as long as $\botcell{\acces_{2m-1}(1,j)}+(-1,1)$ is in~$D$, we set~$\acces_{2m}(1,j) = \lray{\botcell{\acces_K^{2m-1}(1,j)}+(-1,1)}$, otherwise the construction stops;
\item[($iv$)] the set~$\acces_{2m+1}(1,j)$ is the ray~$\vray{\topcell{\acces_K^{2m}(1,j)}+(-1,1)}$.
\end{itemize}
Then we have
\begin{equation}
\label{Eq:ImageCentreDroit}
 h_D^{-1}(\cybis{1,j}) = \sum_{m\ge 0} \sum_{(k,l)\in\acces_m(1,j)} (-1)^m \cybis{k,l}. 
\end{equation}

We define accessible sets of right medial cells in the same way, and we have
\begin{equation}
\label{Eq:ImageCentreGauche}
h_D^{-1}(\cybis{-1,j}) = \sum_{m\ge 0} \sum_{(k,l)\in\acces_m(1,j)} (-1)^m \cybis{k,l}. 
\end{equation}

\end{lemma}

The proof is a computation similar to that in the proof of Lemma~\ref{T:ImageExterne}. We skip it.

\begin{figure}[htbp]
	\begin{center}
	\includegraphics[width=0.5\textwidth]{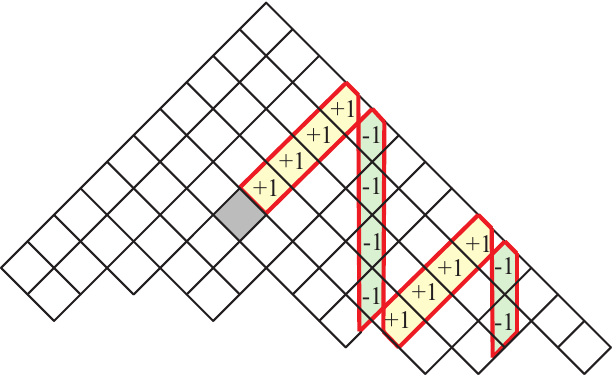}\includegraphics[width=0.5\textwidth]{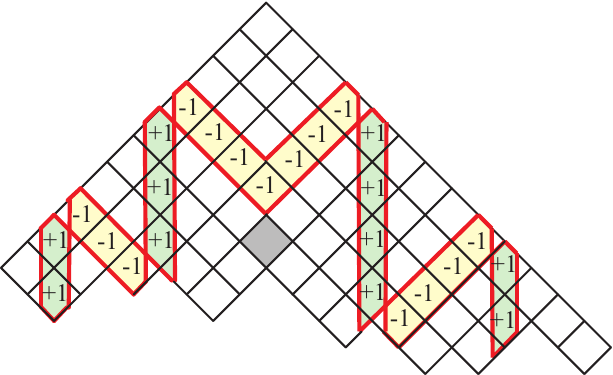}
	\end{center}
	\caption{\small On the left, the image of a left medial cell under the inverse of the monodromy. On the right, the image of a central cell.}
	\label{F:MonodromieCentreDroit}
\end{figure}

In the same vein, we have

\begin{lemma}
\label{T:ImageCentre}
Let $(0,j)$ be a central cell of~$D$ with $j>0$. 
For all $m\ge 1$, define the accessible set~$\acces_m(0,j)$ to be the union $\acces_m(1,j-1) \cup \acces_m(-1,j-1)$.
Then we have
\begin{equation}
\label{Eq:ImageCentre}
h_D^{-1}(\cybis{0,j}) = -\sum_{(k,l)\in\equerre{j-2}}\cybis{k,l} + \sum_{m\ge 1} \sum_{(k,l)\in\acces_K^m(0,j)} (-1)^{m+1} \cybis{k,l}. 
\end{equation}
\end{lemma}

The key point, which has no counterpart in the case of the standard Seifert surface, is as follows. We recall that $n$ stands for the number of cells of the diagram~$D$, and that $b_{-j/2}$ is the vertical coordinate of the bottom cell of the column with abscissa~$-j/2$.

\begin{lemma}
\label{T:DoubleImageCentreDroit}
Let $(1,j)$ be a right medial cell of the diagram~$D$. Then the cycle~$h_D^{-2}(\cybis{1,j})$ is the sum of at most $n$ elementary cycles~$\cybis{k,l}$ all satisfying~$k+l \le -j/2 - b_{-j/2}$.
\end{lemma}

\begin{proof}
Using~\eqref{Eq:ImageCentreDroit}, we obtain
\begin{multline*}
h_D^{-1}(\cybis{1,j}) = \cybis{0,j-1} + \cybis{-1,j-2} +  \sum_{(k,l)\in\acces_K^0(-1,j), k\le -2} \cybis{k,l} \\
 - \sum_{(k,l)\in\acces_K^1(-1,j)} \cybis{k,l} + \sum_{m\ge 2} \sum_{(k,l)\in\acces_K^m(-1,j)} (-1)^m \cybis{k,l}.
\end{multline*}
Comparing~\eqref{Eq:ImageCentreDroit} and~\eqref{Eq:ImageCentre}, and looking at Figure~\ref{F:MonodromieCentreDroit}, one sees that the part of~$h_D^{-1}(\cybis{0,j{-}1})$ in the right columns coincides with~$-h_D^{-1}(\cybis{-1,j{-}2})$. 
Hence both contributions vanish, and we find
\begin{eqnarray*}
h_D^{-2}(\cybis{1,j}) & = & h_D^{-1}(\cybis{0,j{-}1} + \cybis{-1,j{-}2}) +  \sum_{(k,l)\in\acces_K^0(1,j), k\ge 2} h_D^{-1}(\cybis{k,l}) 
\\
& & - \sum_{(k,l)\in\acces_K^1(1,j)} h_D^{-1}(\cybis{k,l}) + \sum_{m\ge 2} \sum_{(k,l)\in\acces_K^m(1,j)} (-1)^m h_D^{-1}(\cybis{k,l}) 
\\
& = & - \sum_{(k,l)\in\acces_K^0(1,j{-}2),k \le -1} \cybis{k,l} + \sum_{(k,l)\in\acces_K^1(1,j{-}2)} \cybis{k,l} 
\\ 
& & + \sum_{m\ge 2} \sum_{(k,l)\in\acces_K^m(1,j{-}2)} (-1)^{m+1} \cybis{k,l}
+ \sum_{(k,l)\in\acces_K^0(1,j), k \le -2} h_D^{-1}(\cybis{k,l}) (\cybis{k,l}) 
\\
& & -\sum_{(k,l)\in\acces_K^1(1,j)} h_D^{-1} + \sum_{m\ge 2} \sum_{(k,l)\in\acces_K^m(1,j)} (-1)^m h_D^{-1}(\cybis{k,l}).
\end{eqnarray*}

Since we have $h_D^{-1}(\cybis{k,l}) = \cybis{k+1, l{-}1}$ for every $k\le -2$, the first two terms in the parenthesis vanish when added to the first one outside, whence
\begin{equation*}
h_D^{-2}(\cybis{1,j}) = 
\sum_{m\ge 2} \sum_{(k,l)\in\acces_K^m(1,j-2)} (-1)^{m+1} \cybis{k,l} + \sum_{m\ge 2} \sum_{(k,l)\in\acces_K^m(1,j)} (-1)^m \cybis{k{+}1,l{-}1}.
\end{equation*}

Depending on whether the iterative constructions of the sets~$\acces^m(1,j{-}2)$ and~$\acces_m(1,j)$ stop or not, some other terms might vanish. In all cases, at most $n$ cycles~$\cybis{k, l}$ remain, all of them lying in the part at the bottom left of~$\acces_m(1,j{-}2)$, and therefore satisfying~$k+l \le -j/2 - b_{-j/2}$. This completes the proof.
\end{proof}

There is one cell whose image has not yet been determined, namely the cell~$(0,0)$. It is the subject of the last lemma of this section.

\begin{lemma}
\label{T:ImageZeroZero}
(See Figure~\ref{F:MonodromieZeroZero}.)
We recursively define accessible sets~$\accesg_m(0,0)$ and~$\accesd_m(0,0)$ as follows\:
\begin{itemize}
\item[($i$)] we put $\accesg_0(0,0) = \accesd_0(0,0) = \vray{0,0}$;
\item[($ii$)] ~$\botcell{\accesg_{2m-2}(0,0)}+(-1,1)$ is a cell of~$D$, we set $\accesg_{2m-1}(0,0) = \lray{\botcell{\accesg_{2m-2}(0,0)}+(-1,1)}$, otherwise the construction stops;
\item[($iii$)] we set $\accesg_{2m}(0,0) = \vray{\topcell{\accesg_{2m-1}(0,0)}+(-1,1)}$;
\item[($iv$)]  similarly, while~$\botcell{\accesd_{2m-2}(0,0)}+(1,1)$ is in~$D$, we set $\accesd_{2m-1}(0,0) = \rray{\botcell{\accesd_{2m-2}(0,0)}+(1,1)}$, otherwise the construction stops;
\item[($v$)]  we set $\accesd_{2m}(0,0) = \vray{\topcell{\accesd_{2m-1}(0,0)}+(1,1)}$.
\end{itemize}
Then we have
\begin{equation*}
 h_D^{-1}(\cybis{0,0}) = \sum_{(k,l)\in\accesg_0(0,0)} \cybis{k,l} + \sum_{m\ge 1} \sum_{(k,l)\in\accesg_m(0,0)} (-1)^m \cybis{k,l} 
+ \sum_{m\ge 1} \sum_{(k,l)\in\accesd_m(0,0)} (-1)^m \cybis{k,l}, 
\end{equation*}
and
\begin{eqnarray*} h_D^{-2}(\cybis{0,0}) = - \sum_{(k,l)\in\equerre{b_0}} \cybis{k,l} &+& \sum_{m\ge 2} \sum_{(k,l)\in\accesg_m(0,0)} (-1)^m \cybis{k{+}1,l{-}1} \\
&+& \sum_{m\ge 2} \sum_{(k,l)\in\accesd_m(0,0)} (-1)^m \cybis{k{-}1,l{-}1}. \end{eqnarray*}
\end{lemma}

Once again the proof is a computation similar to the proof of Lemma~\ref{T:ImageExterne}.

\begin{figure}[htbp]
	\begin{center}
	\includegraphics[width=0.5\textwidth]{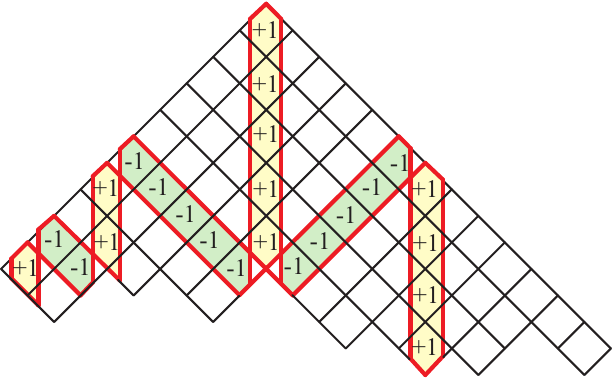}\includegraphics[width=0.5\textwidth]{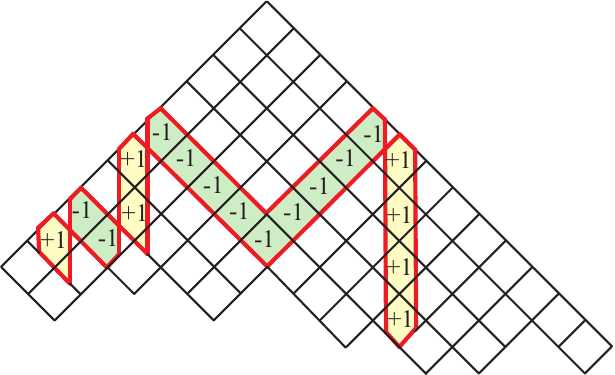}
	\end{center}
	\caption{\small On the left, the image of the cycle~$\cybis{0, 0}$ under the inverse of the monodromy. On the right, its image under the square of the inverse.}
	\label{F:MonodromieZeroZero}
\end{figure}

\section{The spectral radius of the monodromy}
\label{S:Radius}

In this final section, we use the results of Section~\ref{S:SecondPass} to establish a bound on the~$\ell^1$-norm of the inverse of the monodromy. We then deduce bounds for the eigenvalues of the monodromy and, from there, on the zeroes of the Alexander polynomial. We then illustrate the result with a few examples and conclude with questions.

\subsection{Proof of the main result}

We now use the analysis of Section~\ref{S:SecondPass} to precisely describe the iteration of the monodromy on the various cycles according to their position in the Young diagram. It turns out that finitely many patterns only can appear.

\begin{definition}
\label{D:TypeCellules}
Let~$D$ be a Young diagram with $n$ cells, and $\seifbis$ be the associated mixed Seifert surface. We recall that the central column has~$b_0+1$ cells. Then the cycle~$\cybis{i, j}$ associated with the cell with coordinates~$(i, j)$ is said to be of type
\begin{equation*}
\begin{cases}
I_\alpha \mbox{ ({\it resp.\ }$I_\beta$)} &\mbox{if the cell $(i, j)$ is central with $j \le b_0/2$ ({\it resp.\ }$j > b_0/2$),}
\\
I\!I_\alpha \mbox{ ({\it resp.\ }$I\!I_\beta$)} &\mbox{if the cell $(i, j)$ is medial (left or right) with $j \le b_0/2$ ({\it resp.\ }$j > b_0/2$),}
\\
I\!I\!I &\mbox{if we have $\vert i \vert > b_0/4$,}
\\
IV &\mbox{if we have $\vert i \vert \le b_0/4$ and $j - \vert i \vert > b_0/2$,}
\\
X &\mbox{otherwise.}
\end{cases}
\end{equation*}
A cycle associated to a try square~$\equerre{j}$ is said to be of type 

$V_\alpha$ ({\it resp.\ }$V_\beta$) \qquad if $j \le b_0/2$ ({\it resp.\ }$j > b_0/2$).
\end{definition}

\begin{lemma}
\label{T:Base}
Let $D$ be a Young diagram, and~$\seifbis$ be the associated mixed Seifert surface. Then the cycles of type $I\!I_\alpha, I\!I_\beta, I\!I\!I, IV, X, V_\alpha$ and $V_\beta$ form a basis of~$H_1(\seifbis; \Z)$.
\end{lemma}

\begin{proof}
Owing to Proposition~\ref{T:DecompositionMurasugiBis}, the cycles of type $I_\alpha, I_\beta, I\!I_\alpha, I\!I_\beta, I\!I\!I, IV$ and~$X$ form a basis of~$H_1(\seifbis; \Z)$. Since a try square~$\equerre{j}$ is the sum of the cell~$(0,j)$ and of several cells of types~$I\!I\!I$ and~$X$, we keep a basis when replacing the cycle~$\cybis{0,j}$ by~$\sum_{(k,l)\in\equerre{j}}\cybis{k,l}$ for every~$j$.
\end{proof}

\begin{figure}[htbp]
	\begin{center}
	\begin{picture}(340,210)
	\put(0,0){\includegraphics[width=0.8\textwidth]{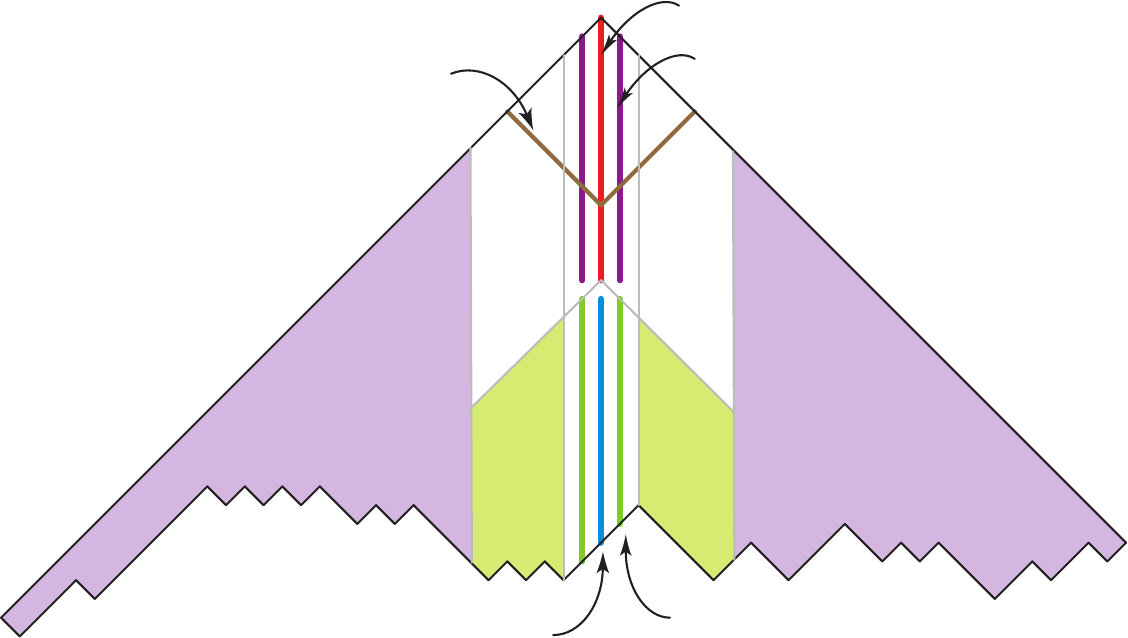}}
	\put(217,200){$I_\alpha$}
	\put(223,184){$I\!I_\alpha$}
	\put(163,-1){$I_\beta$}
	\put(214,2){$I\!I_\beta$}
	\put(131,177){$V_\alpha$}
	\put(100,70){$I\!I\!I$}
	\put(260,70){$I\!I\!I$}
	\put(160,60){$IV$}
	\put(161,120){$X$}
	\put(211,60){$IV$}
	\put(213,120){$X$}
	\end{picture}
	\end{center}
	\caption{\small Cell types in a Young diagram.}
	\label{F:TypeCellules}
\end{figure}
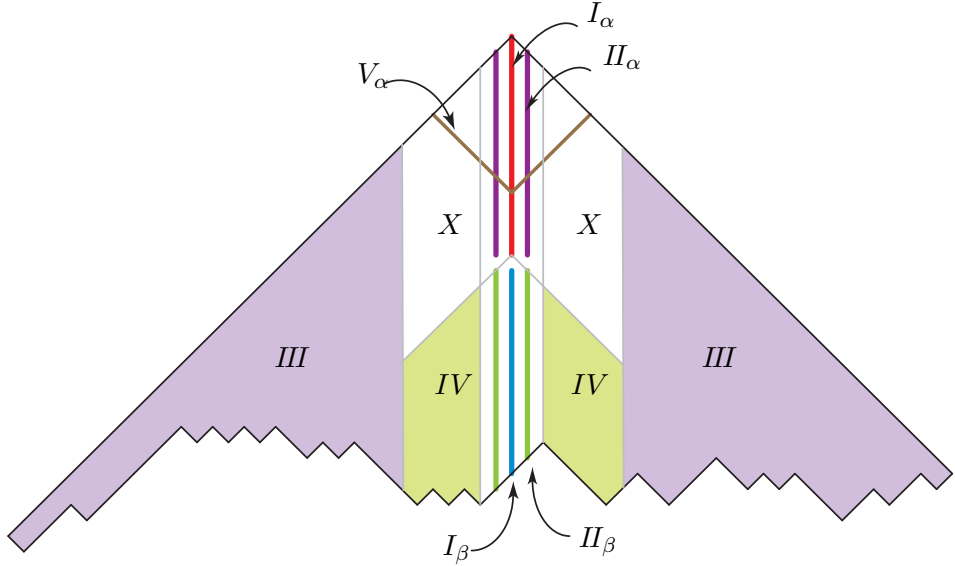

We now collect all information on different types of cells.

\begin{lemma}
\label{T:CroissanceBase}
(See Figure~\ref{F:CheminsCellules}.)
Let $D$ be a Young diagram, $K_D$ be the associated Lorenz knot, $\seifbis$ be the associated mixed Seifert surface, and~$h_D$ be the homological monodromy of~$K_D$. Let $c$ be a basic cycle of~$H_1(\seifbis; \Z)$. Then for $c$ of type
\begin{equation*}
\begin{cases}
I\!I\!I &\mbox{there exists $t_{c} \ge b_0/4$ such that $h_D^{-t_{c}}(c)$ is of type $I\!I_\alpha$ or~$I\!I_\beta$;}
\\
IV &\mbox{there exists $t_{c} \le b_0/4$ such that $h_D^{-t_{c}}([c])$ is of type~$I\!I_\beta$;}
\\
V_\alpha &\mbox{there exists $t_{c} \le b_0/4$ such that $h_D^{-t_{c}}([c])$ is the cycle~$\cybis{0,0}$;}
\\
V_\beta &\mbox{there exists  $t_{c} \ge b_0/4$ such that $h_D^{-t_{c}}([c])$ is the cycle~$\cybis{0,0}$;}
\\ 
I\!I_\alpha &\mbox{the cycle $h_D^{-2}(c)$ is the sum of at most $n$ cycles of type~$I\!I\!I$ or~$IV$;}
\\ 
I\!I_\beta &\mbox{the cycle $h_D^{-2}(c)$ is the sum of at most $n$ cycles of type~$I\!I\!I$.}
\end{cases}
\end{equation*}
Also the cycle $h_D^{-2}(\cybis{0,0})$ is the sum of one cycle of type $V_\beta$ and of at most $n$ cycles of type~$I\!I\!I$.
\end{lemma}

\begin{proof}
If $c$ is of type~$I\!I\!I$ or $IV$, then $c$ corresponds to a peripheral cell, and Lemma~\ref{T:ImagePeripherique} implies that iterated images of~$c$ under~$h_D^{-1}$ go step by step to the center, jumping from one cell to an adjacent one closer to the center of~$D$. The time needed is then prescribed by the distance between the initial cell and the three central columns of~$D$. 

For $c$ is of type~$V_\alpha$ or~$V_\beta$, Lemma~\ref{T:ImageEquerre} describes its iterated images. They are of type~$V$ until they reach the cycle~$\cybis{0,0}$, and the time needed is half the initial height. 
For $c$ of type~$I\!I$, the key-case, Lemma~\ref{T:ImageCentreDroit} describes its image by~$h_D^{-2}$ which, again, has the expected form.
Finally, Lemma~\ref{T:ImageZeroZero} describes the image of~$\cybis{0,0}$.
\end{proof}

We can now state the main result.

\begin{proposition}
\label{T:CheminCellules}
Let $D$ be a Young diagram with $n$ cells, whose central column has $b_0/2$~cells. Let $K_D$ be the associated Lorenz knot. Then all eigenvalues of the homological monodromy of~$K_D$ lie in the annulus~$\left\{z\in\C~\big\vert~ n^{-8/b_0} \le \vert z \vert \le n^{8/b_0} \right\}$.
\end{proposition}

\begin{figure}[htbp]
	\begin{center}
	\begin{picture}(345,170)
	\put(0,0){\includegraphics[width=0.8\textwidth]{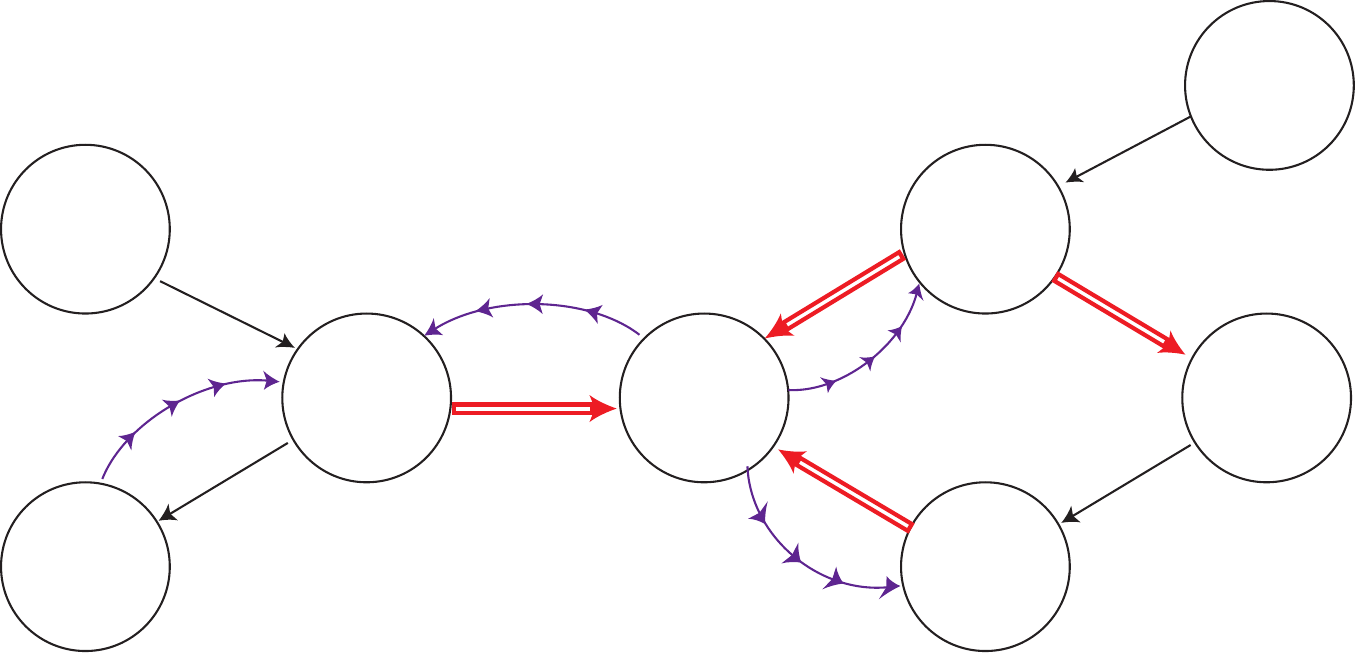}}
	\put(17,107){$V_\alpha$}
	\put(17,18){$V_\beta$}
	\put(251,107.5){$I\!I_\alpha$}
	\put(251,18){$I\!I_\beta$}
	\put(177,63){$I\!I\!I$}
	\put(324.5,62.5){$IV$}
	\put(327,145){$X$}
	\put(84,63){$\cybis{0,0}$}
	\end{picture}
	\end{center}
	\caption{\small Growth of the~$\ell^1$-norm of cycles when the monodromy is iterated. Bold arrows mean that the number of cycles may be multipied by at most~$n$. Small consecutive arrows mean that at least~$b_0/4$ iterations are needed in order to reach the final cell. The key point is that every path containing at least three bold arrows must include a sequence of small arrows.}
	\label{F:CheminsCellules}
\end{figure}
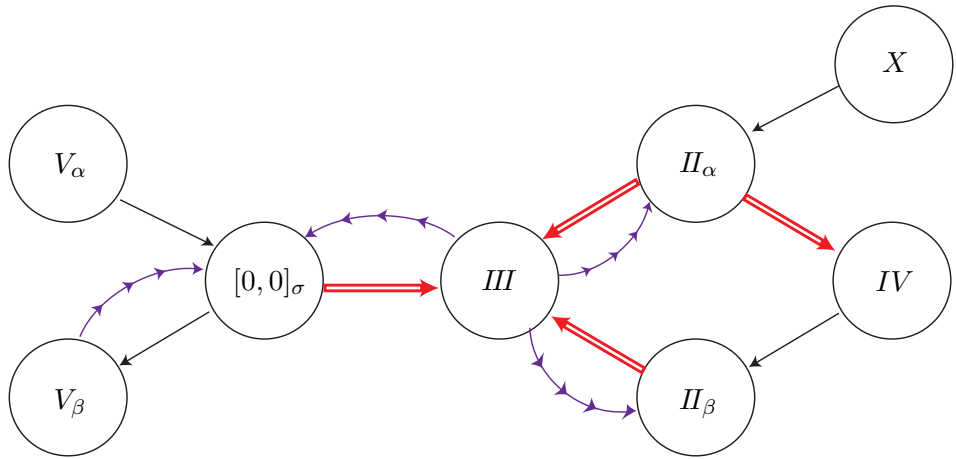

\begin{proof}(See Figure~\ref{F:CheminsCellules})
Let $\seifbis$ be the mixed Seifert surface associated with~$D$.
Let $c$ be a basic cycle of~$H_1(\seifbis; \Z)$. By Lemma~\ref{T:CroissanceBase}, the inverse of the monodromy~$h_D^{-1}$ can increase the $\ell^1$-norm of~$c$ only if $c$ is of type~$I\!I_\alpha, I\!I_\beta$, or if $c$ is the cycle~$\cybis{0,0}$. 
Figure~\ref{F:CheminsCellules} shows that this cannot happen too often: after two iterations, the cycle~$\cybis{0,0}$ and all cycles of type~$I\!I_\beta$ are transformed into at most $n$ cycles of type~$I\!I\!I$, whose norm will not grow in the next~$b_0/4$ iterations of~$h_D^{-1}$.
If $c$ is of type~$I\!I_\alpha$, then there may be two iterations that increase the norm, but the limitation arises: all subsequent cycles are of type~$I\!I\!I$. Therefore, the norm can be multiplied by at most~$n^2$ in a time~$b_0/4$. Then there exists a constant~$A$ so that for every cycle~$c$ and every time~$t$, we have 
\begin{equation*}
\big\vert \log \left(\| h_D^{-t}(c) \|_{{}_1} \right) \big\vert \le \frac{8\log n}{b_0}t + A.
\end{equation*} 
It follows that the eigenvalues of~$h_D^{-1}$ lie in the disk~$\left\{ z~\big\vert~\vert z \vert \le n^{8/b_0}\right\}$. 

On the other hand, the map~$h_D$ preserves the intersection form on~$\seifbis$, which is a symplectic form~\cite[chapter 6]{FarbMargalit}. This implies that the spectrum of~$h_D$ is symmetric with respect to the unit circle~\cite[chapter 1]{Fomenko}. We deduce that the eigenvalues of~$h_D$ lie in the annulus~$\left\{z~\big\vert~ n^{-8/b_0} \le \vert z \vert \le n^{8/b_0} \right\}$.
\end{proof}

We can now conclude.

\begin{proof}[Proof of Theorem~\ref{Theoreme}]
As Lorenz knots are fibered, the zeroes of their Alexander polynomial are eigenvalues of the homological monodromy~\cite{Milnor}. The genus of a Lorenz knot is half the number of cells in every associated Young diagrams~\cite[Corollary 2.4]{EM}, and its braid index is the number of cells of the central column plus one~\cite[main theorem]{FranksWilliams}. The result therefore follows from applying Proposition~\ref{T:CheminCellules}  with $n=2g$ and $b_0=2b-2$.
\end{proof}

\begin{example}
It is known that every algebraic knot is a Lorenz knot~\cite{BW}, and that the zeroes of the Alexander polynomial of an algebraic knot all lie on the unit circle. Therefore they {\it a fortiori} lie in the annulus given by Theorem~\ref{Theoreme}.

The first Lorenz knot whose Alexander polynomial has at least one zero outside the unit circle is the knot associated with the Young diagram~$(4,4,2)$ (see the census~\cite{atlas} of the Lorenz knots with period at most~21). Its genus is~$5$ and its braid index is~$3$. One can indeed check that the $10$~zeroes of the Alexander polynomial satisfy~$20^{-4/2}\le \vert z \vert \le 20^{4/2}$, as prescribed by Theorem~\ref{Theoreme}.

The zeroes of the Alexander polynomials of two generic Lorenz knots with respective braid index~$40$ and~$100$ are displayed on Figure~\ref{F:Racines}. As asserted in Theorem~\ref{Theoreme}, all zeroes lie in some annulus around the unit circle, and the width of the annulus decreases when the braid index increases. Experiments involving large samples of random Young diagrams suggest that the pictures of Figure~\ref{F:Racines} are typical for Lorenz knots of the considered size, {\it i.e.}, that the width of the annulus is roughly determined by the braid index.
\end{example}

\begin{figure}[htbp]
	\includegraphics[width=0.4\textwidth]{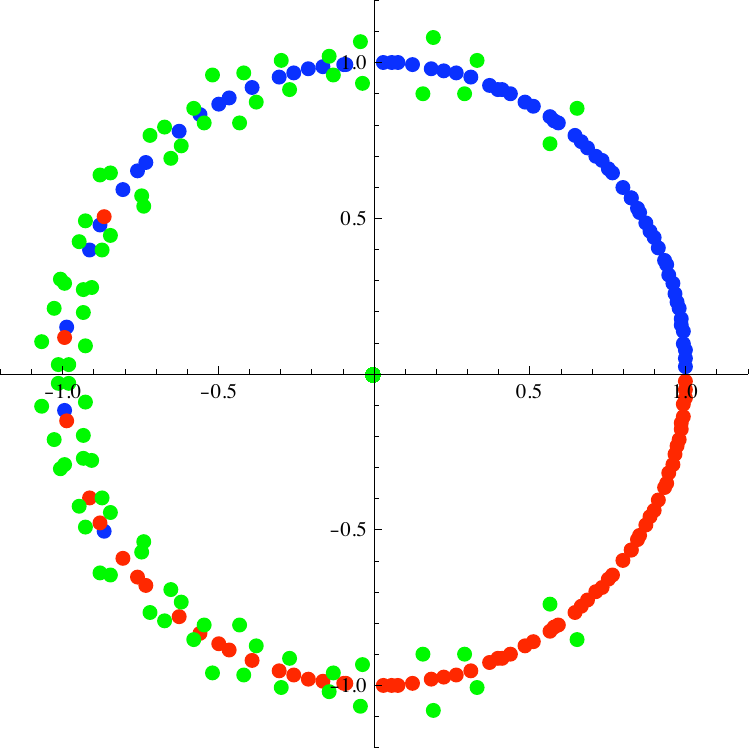}
	\includegraphics[width=0.4\textwidth]{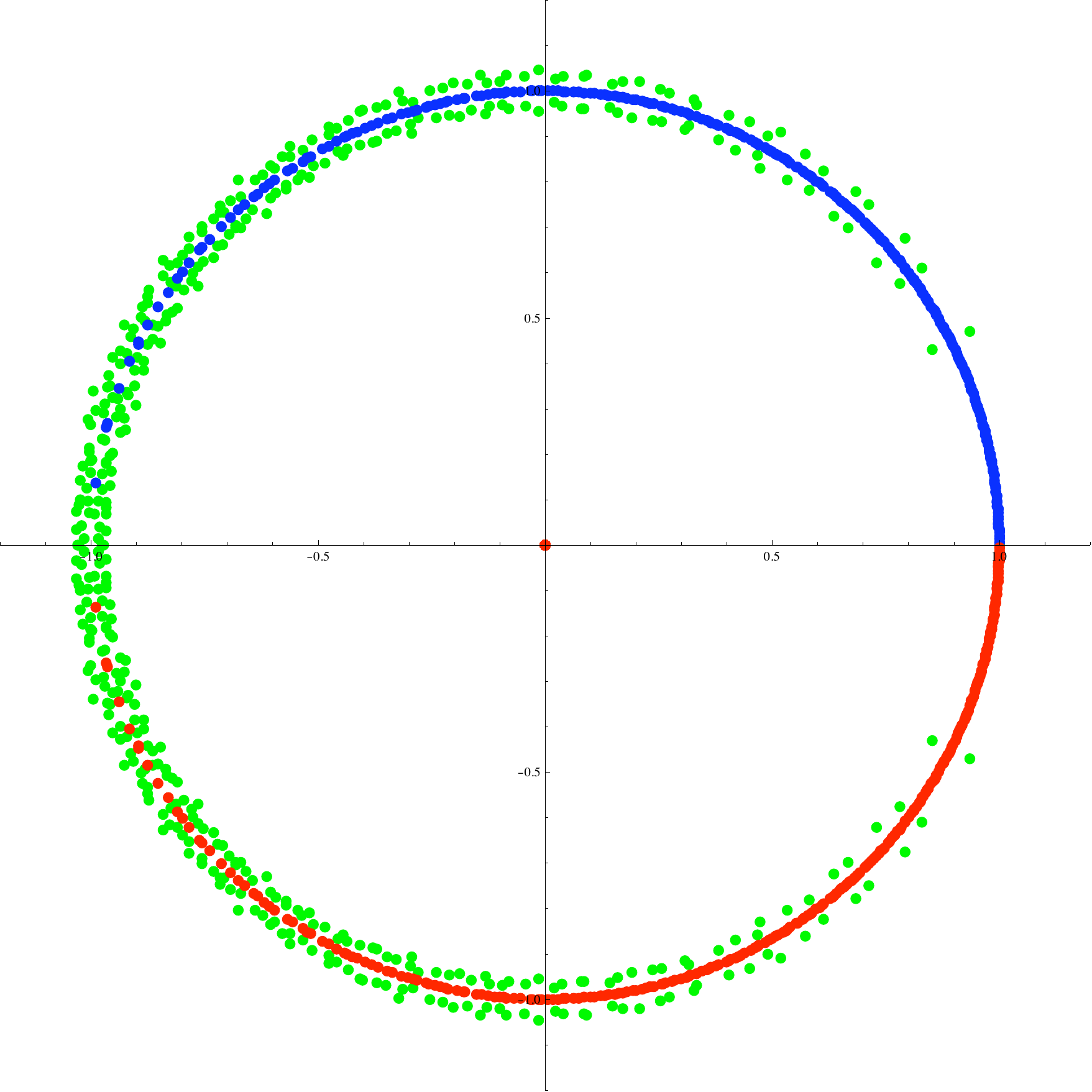}	
	\caption{\small Positions of the zeroes of the Alexander polynomial of two generic Lorenz knots, with braid index~40 and genus~$100$ on the left, and with braid index~100 and genus~$625$ on the right. Green dots correspond to zeroes outside the unit circle, whereas blue and red dots correspond to zeroes on the circle. The annulus containing the zeroes is smaller on the right, as stated by Corollary~\ref{Coro} for typical Lorenz knots.
	}
	\label{F:Racines}
\end{figure}

We now mention two direct consequences of Theorem~\ref{Theoreme}. The first one is a criterion for proving that a knot is not a Lorenz knot. 

\begin{definition}
\label{D:Invariant}
Assume that $K$ is a knot. Let $b$ be its braid index, $g$ be its genus, and $m$ be the maximal modulus of a zero of its Alexander polynomial. Then define the invariant~$r(K)$ as the quotient~${(b-1)\log(m)}/{\log(2g)}$.
\end{definition}

\begin{corollary}
\label{C:NotLorenz}
Let $K$ be a knot. If $r(K)>1$ holds, then $K$ is not a Lorenz knot.
\end{corollary}

Indeed, if $r(K)$ is larger than~$1$, then at least one zero of the Alexander polynomial of~$K$ does not lie in the annulus of Theorem~\ref{Theoreme}, so that the knot cannot be a Lorenz knot.
Using the tables of Livingstone~\cite{Livingston} for knot invariants up to 11 crossings, we could check in this way that 18 out of the 502 knots are not of Lorenz type (according to~\cite{GhysMadrid, atlas}, there are only $8$ Lorenz knots in the above range).

The second consequence of Theorem~\ref{Theoreme} involves the asymptotical position of the zeroes of the Alexander polynomials of a closed orbit of the Lorenz flow, when the length of the orbit goes to infinity. For all~$t$, they are only finitely many closed orbits whose period lies in the interval~$[\,t, (1{+}\epsilon)\, t\,]$. The result states that the longer the orbit, the closer its roots to the unit circle.

\begin{corollary}
\label{Coro}
For every $\epsilon$, there exist $c,c'$ so that the proportion of Lorenz knots with period in the interval~$[\,t, (1{+}\epsilon)\, t\,]$ and with zeroes of the Alexander polynomial all lying in the annulus~$\{ z\in\C \,\big\vert\, c\,t^{-c'/t} \le \vert z \vert \le c\,t^{c'/t} \}$ tends to~$1$ as $t$ goes to infinity.
\end{corollary}

\begin{proof}
There exists a constant $d$ such that a generic length~$t$ orbit of the Lorenz flow crosses the axes of the Lorenz template~(see~\cite[Figure~1]{EM}) at least $dt$~times. Therefore the sum of the width and the height of the Young diagrams associated to generic orbits is at least~$dt$. The braid index of the knot being the size of the largest square sitting inside the Young diagram, it is at least $dt/4$ for a generic period $t$ orbit. The genus of the knot being half the number of cells of the diagram, it is at most~$(dt)^2/8$. Therefore the width of the annulus of Theorem~\ref{Theoreme} associated to generic orbits of the Lorenz flow is at most~$(dt/2)^{8/dt}$.
\end{proof}

\subsection{Further questions}
\label{S:Questions}

We conclude with a few more speculative remarks.

First, by Corollary~\ref{C:NotLorenz}, for every Lorenz knot~$K$, the invariant~$r(K)$ is smaller than~$1$. Numerical experiments indicate that, for Lorenz knots, $r(K)$ might tend to a number close to~$0.15$ when both the braid index and the genus tend to infinity. This suggests that the order of magnitude exhibited in Theorem~\ref{Theoreme} is optimal, but that the constant in the exponent could be improved. More generally, this refers to

\begin{question}
\label{Q:MeilleureBorne}
Is the lower bound of Theorem~\ref{Theoreme} optimal?
\end{question}

A vast abundance of articles~\cite{BBK, Hironaka2, HironakaKin, LT, Penner} are more interested by the dilatation of surface homeomorphisms, which controls the action of the homeomorphism on curves, rather than cycles. Unless the underlying train tracks are orientable -- a very strong restriction pointed out to us by J.\,Birman and not achieved in general by Lorenz knots -- these two invariants do not coincide, the geometrical growth rate being larger, so that our main result does not allow to control the dilatation of the monodromies of the Lorenz knots. However, the key-lemma~\ref{T:ImageInterieure} also holds for curves. Indeed, the image of a curve surrounding an internal cell of a Lorenz knot is a curve surrounding a neighbouring cell. This suggests that curves might also be stretched at a slow rate by the monodromy.

\begin{question}
\label{Q:Geometrical}
Does the dilatation of the monodromy of a Lorenz knot admit bounds similar to Thereom~\ref{Theoreme}?
\end{question}

For generic Lorenz knots, the braid index is of the order of the square root of the genus, so that the value of the parameter~$\log (m)$ is of the order of~$\log (g)/\sqrt{g}$, a value coherent with the above mentioned computer experiments. By contrast, a theorem of Penner~\cite{Penner} says that the dilatation of a pseudo-Anosov map on a surface of genus~$g$ is bounded from below by a function of the order of~$1/g$, an optimal bound. Therefore, the monodromies of generic Lorenz knots do not seem to be pseudo-Anosov homeomorphisms with minimal growth rate. Nevertheless, the situation could be different for particular subfamilies:

\begin{question}
Is there an infinite family of Lorenz knots admitting monodromies with a homological growth rate of the order of~$1/g$?
\end{question}

In a totally different direction, Figure~\ref{F:Alea} shows the location of the zeroes of the Alexander polynomial of random positive braids with braid index $3,4$, and $5$, respectively, and of a non-positive random braid. When the braid index has a fixed value~$b$ and we consider positive braids with increasing length, the majority of the roots seem to accumulate on a specific curve, which depends on the braid index and on the probabilities of the generators~$\sigma_i$, and which is smooth except at some singular points whose arguments are multiples of~$2\pi/b$. This situation contrasts with Theorem~\ref{Theoreme} radically, and no explanation is known so far.

\begin{figure}[htbp]
	\includegraphics[width=0.2\textwidth]{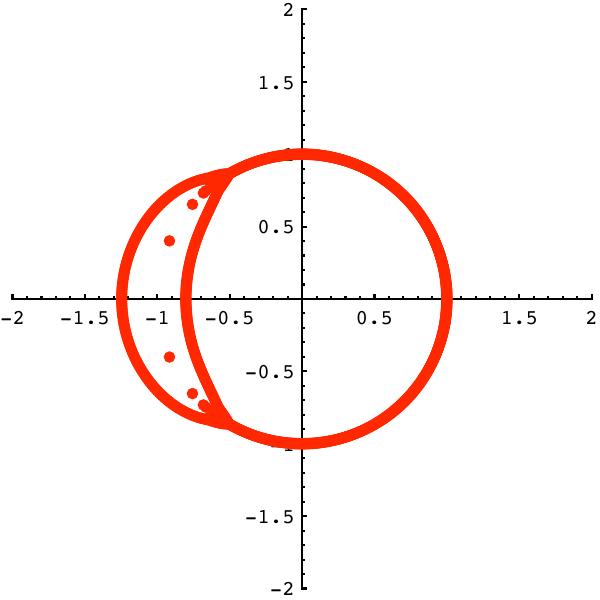}
	\includegraphics[width=0.2\textwidth]{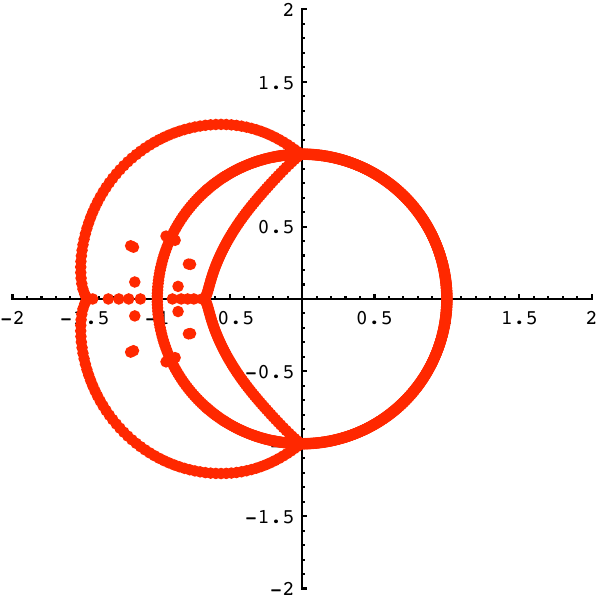}
	\includegraphics[width=0.2\textwidth]{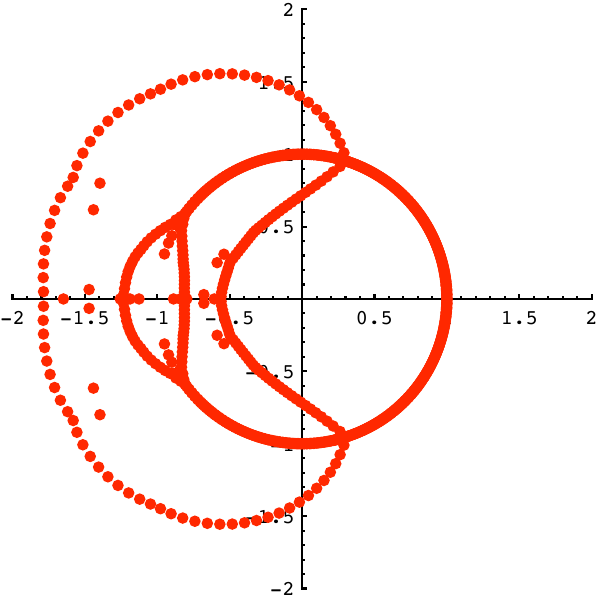}
	\includegraphics[width=0.2\textwidth]{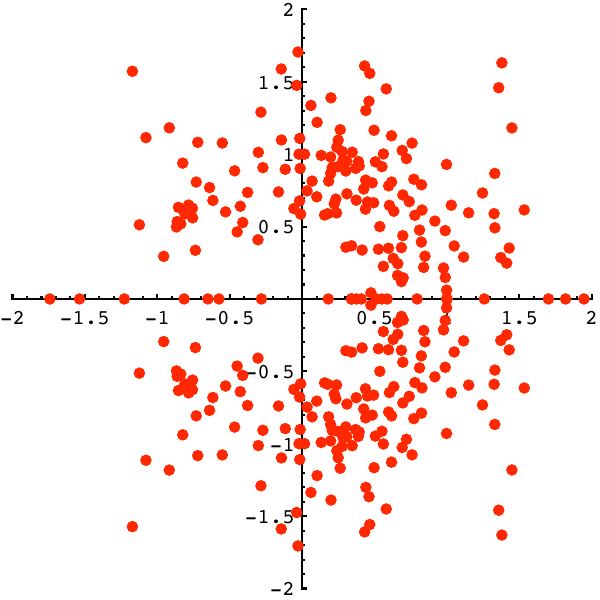}
	\caption{\small 	The zeroes of the Alexander polynomial of random braids of length~200. From left to right: a positive braid of index~3, a positive braid of index~4, a positive braid of index~5, and a braid with both positive and negative crossings of index~5; in each case the generators~$\sigma_i$ are chosen with a uniform distribution. In the first three cases, the zeroes seem to accumulate on very particular curves.
	}
	\label{F:Alea}
\end{figure}

Via the Burau representation, the Alexander polynomial of the closure of a length $\ell$ braid of index~$n$ can be expressed as the determinant of a matrix of the form $M_{i_1} M_{i_2} \cdots M_{i_\ell} - I_n$, where $M_1 , \ldots, M_n$ are the matrices of~$\mathcal{M}_n(\Z[z])$ that correspond to the length one braids~$\sigma_i$. Then, a complex number~$z$ is a zero of the Alexander polynomial if and only if $1$ is an eigenvalue of the corresponding product of matrices. Looking at Figure~25 leads to

\begin{question}
Let $M_1, \ldots, M_m$ be fixed invertible matrices in~$\mathcal{M}_n(\Z[z])$. Form the product $\Pi_\ell(z) = M_{i_1} M_{i_2} \cdots M_{i_\ell}$ where $i_1, \ldots, i_\ell$ are independent and equidistributed random variables in $\{1, \ldots, m\}$, and let $D_\ell$ be the set of~$z$ such that $1$ is an eigenvalue of~$\Pi_\ell(z)$. When does $D_\ell$ admit a Hausdorff limit? And, if so, what does the limit look like? 
\end{question}


\end{document}